\documentclass[a4paper, 11pt]{article}
\usepackage{amsmath,amssymb,amsthm}
\usepackage[scaled]{helvet}
\usepackage{braket}
\usepackage{cases}
\usepackage{mathrsfs}
\usepackage{amsfonts}
\allowdisplaybreaks[4]
\usepackage{amscd}
\usepackage[all]{xy}
\usepackage{comment}
\usepackage[top=30truemm,bottom=30truemm,left=25truemm,right=25truemm]{geometry}
\usepackage{ascmac}
\usepackage[pdftex]{xcolor}
\usepackage{tikz,float}
\usepackage[utf8]{inputenc}
\usepackage[T1]{fontenc}
\usepackage{stmaryrd}
\usepackage{mathtools}

\usepackage{3dplot} 

\usetikzlibrary{shapes,arrows,positioning}

\usetikzlibrary{shapes.misc}
\usetikzlibrary{decorations.markings}
\tikzset{cross/.style={cross out, draw=black, minimum size=2*(#1-\pgflinewidth), inner sep=0pt, outer sep=0pt},
cross/.default={1pt}}
\tikzset{->-/.style={decoration={
  markings,
  mark=at position #1 with {\arrow[scale=1.5]{>}}},postaction={decorate}}}

\tikzset{-<-/.style={decoration={
  markings,
  mark=at position #1 with {\arrow[scale=1.5]{<}}},postaction={decorate}}}

\tikzset{midstealth/.style={decoration={
  markings,
  mark=at position #1 with {\arrow{stealth}}},postaction={decorate}}}
\newtheoremstyle{break}
  {}%
  {}%
  {\itshape}
  {}%
  {\bfseries}
  {.}%
  {\newline}%
  {}%

 \makeatletter
    
    \@addtoreset{equation}{section}
  \makeatother

\theoremstyle{plain}
\newtheorem{thm}{Theorem}[subsection]
\newtheorem{lem}[thm]{Lemma}
\newtheorem{lemma}[thm]{Lemma}
\newtheorem{prop}[thm]{Proposition}
\newtheorem{cor}[thm]{Corollary}
\newtheorem{exa}[thm]{Example}

\newtheorem{rem}[thm]{Remark}

\DeclareMathOperator{\rank}{rank}

\DeclareMathOperator{\vol}{vol}

\DeclareMathOperator{\im}{Im}
\DeclareMathOperator{\re}{Re}
\DeclareMathOperator{\Ker}{Ker}
\DeclareMathOperator{\Hom}{Hom}

\DeclareMathOperator{\Homo}{H}

\DeclareMathOperator{\sol}{Sol}
\DeclareMathOperator{\diag}{diag}

\DeclareMathOperator{\cone}{cone}

\newcommand{\barsigma}{\overline{\sigma}}
\newcommand{\PP}{\mathbb{P}}
\newcommand{\R}{\mathbb{R}}

\newcommand{\C}{\mathbb{C}}

\newcommand{\Q}{\mathbb{Q}}
\newcommand{\Z}{\mathbb{Z}}

\newcommand{\ii}{\sqrt{-1}}
\newcommand{\s}{\sigma}

\newcommand{\bs}{\barsigma}

\newcommand{\pa}{\partial}
\newcommand{\ve}{\varepsilon}
\newcommand{\GG}{GG}

\newcommand{\OM}{\mathscr{O}\hspace{-.4em}\mathscr{M}}

\def\change#1{{ #1}}

\title{Global analysis of GG systems}
\author{Saiei-Jaeyeong Matsubara-Heo\footnote{Graduate School of Science, Kobe  University, 1-1 Rokkodai, Nada-ku, Kobe 657-8501, Japan.\newline e-mail: \texttt{saiei@math.kobe-u.ac.jp}}}

\begin{document}

\date{}
\maketitle

\begin{abstract}      
This paper deals with some analytic aspects of GG system introduced by I.M.Gelfand and M.I.Graev: we compute the dimension of the solution space of GG system over the field of meromorphic functions periodic with respect to a lattice. We describe the monodromy invariant subspace of the solution space. We give a connection formula between a pair of bases consisting of $\Gamma$-series solutions of GG  system associated to a pair of regular triangulations adjacent to each other in the secondary fan.
\end{abstract}

\maketitle

\section{Introduction}
In the 80's and 90's, the general study of hypergeometric functions made progress in a series of papers by I.M.Gelfand, M.M.Kapranov, and A.V.Zelevinsky (\cite{GKZToral}, \cite{GKZEuler}, \cite{GKZbook}). One of the new perspectives of their study is that there is a combinatorial structure of convex polytopes behind hypergeoemetric systems. GKZ system is a system of linear partial differential equations determined by two inputs: an $n\times N$ ($n<N$) integer matrix $A=(a_{ij})$ and a parameter vector $c\in\C^{n}$. GKZ system $M_A(c)$ is defined by
\begin{subnumcases}{M_A(c):}
E_i\cdot f(z)=0 &($i=1,\dots, n$)\label{EulerEq}\nonumber\\
\Box_u\cdot f(z)\hspace{-0.8mm}=0& $\left(u={}^t(u_1,\dots,u_N)\in L_A=\Ker(A\times:\Z^{N}\rightarrow\Z^{n})\right)$,\label{ultrahyperbolic}\nonumber
\end{subnumcases}
where $E_i$ and $\Box_u$ are differential operators defined by 

\begin{equation}\label{HGOperators}
E_i=\sum_{j=1}^Na_{ij}z_j\frac{\partial}{\partial z_j}+c_i,\;\;\;
\Box_u=\prod_{u_j>0}\left(\frac{\partial}{\partial z_j}\right)^{u_j}-\prod_{u_j<0}\left(\frac{\partial}{\partial z_j}\right)^{-u_j}.
\end{equation}

\noindent
We write ${\bf a}(j)$ for the $j$-th column vector of $A$. A fundamental property that GKZ system $M_A(c)$ enjoys is the holonomicity (\cite[THEOREM 3.9]{Adolphson}) and in particular, the finiteness of the dimension of the solution space. Many classical hypergeoemetric systems are realized as particular examples of GKZ system for special choices of the configuration matrices $A$. GKZ system also appears naturally in various contexts of applications such as mirror symmetry (\cite{Batyrev}, \cite{StienstraMirror}) and algebraic statistics (\cite{KTT}, \cite{GKTT}).   


An important open question is to understand the monodromy representation of GKZ system. More concretely, one hopes to find explicit monodromy matrices with respect to a given basis of solutions and a set of generators of the fundamental group of the complement of the singular locus. However, there are two difficulties: Firstly, GKZ system $M_A(c)$ behaves in a singular way when the parameter $c$ takes a special value. A typical example of this phenomenon is the discontinuity of the rank (\cite{SST}, \cite{MMW}). Secondly, the structure of the fundamental group is not yet fully understood. It is known that the singular locus of the GKZ system is (contained in) a product of principal $A$-determinants (\cite{BerkeschMatusevichWalther}). However, the degree of the defining equation of the principal $A$-determinant can be too big even if $A$ is relatively small (\cite[Chapter 1]{GKZLeningrad}). Note that there is an interesting recent work \cite{FM} on the fundamental group of the complement of the principal $A$-determinant.


Despite these difficulties, several people made progress on the global analysis of GKZ system. Let us review some of the preceding results on the monodromy of GKZ system when the parameter $c$ is generic. The genericity of the parameter $c$ plays a key role when we take a specific basis of the solution space. Let us also assume that the GKZ system $M_A(c)$ in question is regular holonomic. Note that $M_A(c)$ is regular holonomic if and only if the configuration matrix $A$ is homogeneous, i.e., there is a linear function $\phi:\Z^n\rightarrow \Z$ such that $\phi({\bf a}(j))=1$ for any $j=1,\dots,N$ (\cite{Hotta}, \cite{SchulzeWaltherIrregularity}, \cite{FernandezFernandez}). Recall that the totality of regular polyhedral subdivisions has the structure of a convex polyhedral fan, which is called the secondary fan denoted by ${\rm Fan}(A)$ in this paper (\cite[Chapter 7]{GKZbook}). To each regular triangulation $T$, we can associate a basis $\Phi_T$ of series solutions of $M_A(c)$ convergent on a suitable open subset $U_T$ of $\C^N$ (\cite{GKZToral}, \cite{SST}). When the configuration matrix $A$ comes from a product of simplices $\Delta_1\times \Delta_{n-1}$, Mutsumi Saito and Nobuki Takayama gave an explicit description of the connection matrices among bases $\Phi_T$ in \cite{SaitoTakayama}. They fully utilized the fact that the open cones of the secondary fan and elements of the permutation group $\mathfrak{S}_n$ are in one-to-one correspondence. In this way, the connection matrices they obtained are labeled by the permutation group $\mathfrak{S}_n$. Another approach to the connection problem was proposed by Frits Beukers in \cite{BeukersMellinBarnes}. Under a suitable assumption on the configuration matrix $A$, he constructed a special basis $\Phi$ of solutions in terms of Mellin-Barnes integral. Taking residues of the integrand in one direction, he succeeded in showing a relation between the bases $\Phi$ and $\Phi_T$ for each regular triangulation $T$, which is utilized to compute the connection matrices among bases $\Phi_T$. Though this method is useful, a basis consisting of Mellin-Barnes integral does not always exist. In fact, GKZ systems corresponding to Appell-Lauricella's $F_A,F_B$ and $F_D$ have such a basis while $F_C$ does not.

In this paper, we address the problem of computing connection matrices among bases $\Phi_T$ of GKZ system $M_A(c)$ with a generic parameter from another point of view. Namely, we give a combinatorial description of the connection matrix between $\Phi_T$ and $\Phi_{T^\prime}$ for any regular triangulations $T$ and $T^\prime$. For this purpose, we only need to discuss the case when $T$ is adjacent to $T^\prime$, that is, when the cone of the secondary fan corresponding to $T$ shares a facet with that corresponding to $T^\prime$. When $T$ is adjacent to $T^\prime$, $T$ and $T^\prime$ are related to each other by a combinatorial operation called {\it modification} (\cite[Chap.7, \S2.C]{GKZbook}). Therefore, it is natural to expect that the connection matrix between $\Phi_T$ and $\Phi_{T^\prime}$ is described in terms of modification. A useful method of performing an analytic continuation of a solution of a system of partial differential equations with regular singularities is the method of boundary value problem (\cite{Heckman}, \cite{KashiwaraOshima}). The use of this method in the context of GKZ system has already been indicated in \cite[Theorem 1.3]{SaitoTakayama}. When $T$ is adjacent to $T^\prime$, \cite[Theorem 1.4]{SaitoTakayama} shows that the boundary value problem along a particular coordinate subspace is naturally defined and the connection matrix for boundary values gives rise to that between $\Phi_T$ and $\Phi_{T^\prime}$. 


In this paper, we provide another perspective to the boundary value problem by employing the viewpoint of GG system, a system of linear partial difference-differential equations on $\C^N\times\C^n$ (\cite{GelfandGraevGG}):

\begin{subnumcases}{\GG(A):}
E_i\cdot f(z;c)=0 &($i=1,\cdots, n$)\label{EulerEq2}\\
\frac{\partial}{\partial z_j} f(z;c)\hspace{-0.8mm}=f(z;c+{\bf a}(j))& ($j=1,\dots,N$).\label{contiguity}
\end{subnumcases}
Solutions of GKZ system $M_A(c)$ with a generic parameter $c$ are naturally regarded as those of GG system (\cite[Theorem 4]{GelfandGraevGG}). A crucial difference, however, is that the parameter vector $c$ is now regarded as an independent variable. For this reason, it is natural to regard the solution space $\sol_{GG(A)}$ of $GG(A)$ as a vector space over a \change{fraction field of exponential polynomials in $c$}. In this sense, $\Phi_T$ is also a basis of $\sol_{GG(A)}$. Let us also remark that the viewpoint of GG system naturally encodes the contiguity structure of hypergeometric functions. Therefore, saying that a function $f(z;c)$ is a solution of GG system has more information than saying that it is a solution of a GKZ system.

In the setting of GG system, we can prove a unique solvability of the boundary value problem (Theorem \ref{prop:BVP}) on the level of formal solutions, which plays the role of \cite[Theorem 1.3]{SaitoTakayama}. Thanks to the appearance of the variable $c$, we can also describe the inverse of the boundary value map as a difference operator of infinite order $\mathfrak{D}$. 
Since this operator $\mathfrak{D}$ may produce a divergent series, we should establish some estimates to justify our argument which will be carried out in \S\ref{sec:connection}. We also construct a path of analytic continuation on which our estimate is valid. The main result Theorem \ref{thm:main} is given in \S\ref{subsec:conn}.

Another usage of the boundary value problem is the construction of a monodromy invariant subspace. It was conjectured in \cite{FF2} that an irregular GKZ system has a monodromy invariant subspace associated to any facet of its Newton polytope which does not contain the origin. In \S2, we prove this conjecture on the level of GG system.

This paper consists of two parts. \S\ref{sec:2} is devoted to a preliminary study of GG system. Throughout this section, we do not assume that the configuration matrix is homogeneous. In \S\ref{subsec:2.1}, we deal with series solutions \change{and integral representations} of GG systems. After we recall the relation between series solutions and the combinatorics of the secondary fan, we \change{prove an isomorphism between a homology group and the solution space of GG system and} provide the formula of the dimension of the solution space (Theorem \ref{thm:2.3}). This formula can be seen as an analytic counterpart of \cite[Theorem 3]{OharaTakayama}. In \S\ref{subsec:2.2}, we establish the unique solvabillity of the boundary value problem (Theorem \ref{prop:BVP}). \S\ref{subsec:2.3} is independent of the discussion of the next section, but we give an interesting application of Theorem \ref{prop:BVP}. Namely, we prove that the combinatorics of the Newton polytope gives rise to a decomposition of the solution space $\sol_{GG(A)}$ into monodromy invariant subspaces (Theorem \ref{thm:2.7}). \S\ref{sec:connection} is devoted to the formulation and the proof of the main theorem of this paper, a connection formula. In this section, we assume that $A$ is homogeneous. In \S\ref{subsec:3.2}, we recall the notion of modification and the well-known method of analytic continuation by means of Mellin-Barnes integral in our setting (\cite[Chap.4]{SlaterBook}). In \S\ref{subsec:3.1}, we give some combinatorial lemmata related to the secondary fan. In \S\ref{subsec:3.3}, we establish estimates of difference operators utilizing an integral operator called Erd\'elyi-Kober operator. In \S\ref{subsec:conn}, we prove the connection formula in Theorem \ref{thm:main}.

Throughout this paper, we use the following notation: for any vectors ${\bf a}=(a_1,\dots,a_n),{\bf b}=(b_1,\dots,b_n)\in\C^n$, we set $|{\bf a}|=a_1+\cdots+a_n$, $e^{2\pi\ii{\bf a}}=(e^{2\pi\ii a_1},\dots,e^{2\pi\ii a_n})$, ${\bf a}{\bf b}=(a_1b_1,\dots,a_nb_n)$, and ${\bf a}^{\bf b}=a_1^{b_1}\cdots a_n^{b_n}$. We write ${\rm diag}(a_i)_{i=1}^r$ for an $n\times n$ diagonal matrix whose diagonal entries are given by $a_i$. For any univariate function $F$, we write $F({\bf a})$ for the product $F(a_1)\cdots F(a_n)$.  For any $1\times n$ row vector $z$ and any $n\times m$ matrix $B=({\bf b}(1)|\cdots|{\bf b}(m))$, we write $z^B$ for the row vector $(z^{{\bf b}(1)},\dots,z^{{\bf b}(m)})$. The symbol $\Z B$ denotes the lattice in $\Z^n$ generated by column vectors of $B$. If $b_{ij}$ is the $(i,j)$-entry of the matrix $B$, we write $(b_{ij})_{i,j}$ for the matrix $B$.

\section{Basic properties of GG system}\label{sec:2}
In this section, we establish some basic properties of GG system. Throughout this section, $A$ denotes an $n\times N$ integer matrix with $n<N$. For any subset $\s\subset \{ 1,\dots,N\}$, we write $A_{\s}$ for the matrix whose column vectors consist precisely of $j$-th column vectors of $A$ for all $j\in\s$. We set $\bs:=\{ 1,\dots,N\}\setminus\s$. Note that $A_{\s}$ (resp. $A_{\s}^{-1}$) is regarded as a matrix whose rows are labeled by the set $\{ 1,\dots,n\}$ (resp. \change{labeled by the set }$\s$) and whose columns are labeled by the set $\s$ (resp. \change{labeled by the set} $\{ 1,\dots,n\}$). We say that $\s$ is a simplex if its cardinality $|\s|$ is $n$ and $\det A_{\s}\neq 0$. For any simplex $\s$, $i\in\s$ and a column vector $v\in\C^{n}$, we write  $p_{\s i}(v)$ for the $i$-th entry of the vector $A^{-1}_\s v$.

\subsection{Integral representations and the dimension of the solution space}\label{subsec:2.1}


In this subsection, we assume that the column vectors of $A$ generate the lattice $\Z^n$. We fix a simplex $\s\subset \{ 1,\dots,N\}$. For any partition $\s=\s^u\sqcup\s^d$, we put
\begin{equation}\label{eqn:2.1}
\psi_{\s^d}^{\s^u}(z;c)=
\sum_{{\bf m}\in\Z^{\bs}_{\geq 0}}
\frac{
\displaystyle
\prod_{i\in\s^u}\Gamma(p_{\s i}(c+A_{\bs}{\bf m}))
}{
\displaystyle
\prod_{i\in\s^d}\Gamma(1-p_{\s i}(c+A_{\bar{\sigma}}{\bf m})){\bf m!}
}
\prod_{i\in\s^u}(e^{\pi\ii}z_i)^{-p_{\s i}(c+A_{\bar{\sigma}}{\bf m})}
\prod_{i\in\s^d}z_i^{-p_{\s i}(c+A_{\bar{\sigma}}{\bf m})}
z_{\bar{\sigma}}^{\bf m}.
\end{equation}

\noindent
For any choice of $\tilde{\bf k}\in\Z^{\s}$, we set $\psi_{\s^d,\tilde{\bf k}}^{\s^u}(z;c):=\psi_{\s^d}^{\s^u}(e^{2\pi\ii\tilde{\bf k}}z_\s,z_{\bs};c)$.

\begin{prop}
$\psi_{\s^d,\tilde{\bf k}}^{\s^u}(z;c)$ is a formal solution of $\GG(A)$.
\end{prop}

\begin{proof}
It is enough to prove the proposition for $\psi_{\s^d}^{\s^u}(z;c)$. For any $t\in(\C^*)^{n\times 1},$ we can easily see that $\psi_{\s^d}^{\s^u}(t^A\cdot z;c)=t^{-c}\psi_{\s^d}^{\s^u}(z;c)$. Taking the partial derivative in the variable $t_i$ and substituing $t_1=\cdots =t_n=1$, we obtain the equation (\ref{EulerEq}). 

Suppose that $j\in\bs$. We write ${\bf e}_j\in\Z^{\bs}$ for the vector whose entries are 0 except for the $j$-th entry and $j$-th entry of which is $1$.  Then,
\begin{align}
  &\frac{\pa}{\pa z_j}\psi_{\s^d}^{\s^u}(z;c)\\
=&\sum_{{\bf m}-{\bf e}_j\in\Z^{\bs}_{\geq 0}}
\frac{
\displaystyle
\prod_{i\in\s^u}\Gamma(p_{\s i}(c+{\bf a}(j)+A_{\bs}({\bf m}-{\bf e}_j)))
}{
\displaystyle
\prod_{i\in\s^d}\Gamma(1-p_{\s i}( c+{\bf a}(j)+A_{\bar{\sigma}}({\bf m}-{\bf e}_j)))({\bf m}-{\bf e}_j)!
}
\times\nonumber\\
&\prod_{i\in\s^u}(e^{\pi\ii}z_i)^{-p_{\s i}(c+{\bf a}(j)+A_{\bar{\sigma}}({\bf m}-{\bf e}_j))}\prod_{i\in\s^d}z_i^{-p_{\s i}(c+{\bf a}(j)+A_{\bar{\sigma}}({\bf m}-{\bf e}_j))}z_{\bar{\sigma}}^{{\bf m}-{\bf e}_j}\\
=&\psi_{\s^d}^{\s^u}(z;c+{\bf a}(j)).
\end{align}

\noindent
Next, let us fix any $i_0\in\s^u$. We have
\begin{align}
  &\frac{\pa}{\pa z_{i_0}}\psi_{\s^d}^{\s^u}(z;c)\\
=&\sum_{{\bf m}\in\Z^{\bs}_{\geq 0}}
\frac{
\displaystyle
\prod_{i\in\s^u}\Gamma(p_{\s i}(c+{\bf a}(i_0)+A_{\bs}{\bf m}))
}{
\displaystyle
\prod_{i\in\s^d}\Gamma(1-p_{\s i}(c+A_{\bar{\sigma}}{\bf m})){\bf m!}
}
\prod_{i\in\s^u}
(e^{\pi\ii}z_i)^{-p_{\s i}(c+{\bf a}(i_0)+A_{\bar{\sigma}}{\bf m})}
\prod_{i\in\s^d}
z_i^{-p_{\s i}(c+{\bf a}(i_0)+A_{\bar{\sigma}}{\bf m})}
z_{\bar{\sigma}}^{\bf m}\\
=&\psi_{\s^d}^{\s^u}(z;c+{\bf a}(i_0)).
\end{align}

\noindent
The case when $i_0\in\s^d$ can be proved in a similar way.
\end{proof}

We briefly recall the definition of a regular polyhedral subdivision \change{(\cite[Chapter CHAPTER 7]{GKZbook}, \cite[CHAPTER8]{SturmfelsLecture})}. In general, for any subset $\sigma$ of $\{1,\dots,N\},$ we write $\cone(\sigma)$ for the positive span of \change{$\{{\bf a}(i)\}_{i\in\s}$,} i.e., $\cone(\sigma)=\sum_{i\in\sigma}\R_{\geq 0}{\bf a}(i).$ We often identify a subset $\sigma\subset\{1,\dots,N\}$ with the corresponding set of vectors $\{{\bf a}(i)\}_{i\in\sigma}$ or with the set $\cone(\s)$. A collection \change{$S$} of subsets of $\{1,\dots,N\}$ is called a \change{\it polyhedral subdivision} if $\{\cone(\sigma)\mid \sigma\in S\}$ is the set of cones in a \change{polyhedral} fan whose support equals $\cone(A)$. We write $(\Z^{N})^\vee$ for the dual lattice of $\Z^{N}$.
We write $\pi_A:(\Z^{N})^\vee\rightarrow L_A^\vee$ for the dual of the natural inclusion $L_A\hookrightarrow \Z^{N}$ where $L_A^\vee$ is the dual lattice $\Hom_{\Z}(L_A,\Z)$. By abuse of notation, we still write $\pi_A:(\R^{N})^\vee\rightarrow L_A^\vee\underset{\Z}{\otimes}\R$ for the linear map $\pi_A\underset{\Z}{\otimes}{\rm id}_{\R}$ where ${\rm id}_{\R}:\R\rightarrow\R$ is the identity map. For any cone $C\subset \R^N$, the symbol $C^\vee$ denotes its dual cone.
We often identify $(\R^N)^\vee$ with the set of row vectors via dot product. Then, for any choice of a vector $\omega\in\pi_A^{-1}\left(\pi_A((\R^{N}_{\geq 0})^\vee)\right),$ we can define a polyhedral subdivision $S(\omega)$ as follows: A subset $\sigma\subset\{1,\dots,N\}$ belongs to $S(\omega)$ if there exists a row vector ${\bf n}\in\R^{1\times n}$ such that ${\bf n}\cdot{\bf a}(i)=\omega_i$ if $i\in\sigma$ and ${\bf n}\cdot{\bf a}(j)<\omega_j$ if $ j\in\barsigma$. A polyhedral subdivision $S$ is called a {\it regular polyhedral subdivision} if $S=S(\omega)$ for some $\omega.$ Given a regular polyhedral subdivision $S$, we write $C_S\subset (\R^N)^\vee$ for the cone consisting of vectors $\omega$ such that $S(\omega)=S$. If any maximal (with respect to inclusion) element $\s$ of a regular polyhedral subdivision $T$ is a simplex, we call $T$ a {\it regular triangulation}. For a fixed regular triangulation $T$, we say that the parameter vector $c$ is {\it very generic} if $A_\s^{-1}(c+{\bf k})$ has no integral entry for any  simplex $\sigma\in T$ and any ${\bf k}\in\Z^n$. Let us put $H_\sigma=\{ j\in\{ 1,\dots, N \}\mid |A_\sigma^{-1}{\bf a}(j)|=1\}$. We set 
\begin{equation}
U_\sigma=\left\{z\in(\C^*)^N\mid {\rm abs}\left(z_\sigma^{-A_\sigma^{-1}{\bf a}(j)}z_{j}\right)<R, \text{for all } a(j)\in H_\sigma\setminus\sigma\right\},
\end{equation}
where $R>0$ is a small positive real number and abs stands for the absolute value. Recall that a regular triangulation $T$ is said to be {\it convergent} if for any $n$-simplex $\s\in T$ and for any $j\in \bs$, one has the inequality $|A_\sigma^{-1}{\bf a}(j)|\leq 1$ (\cite[Definition 5.2]{MatsubaraEulerLaplace}). \change{Note that a convergent regular triangulation always exists.}

\begin{prop}\label{prop:2.2}
Fix a convergent regular triangulation $T$, a partition $\s=\s^u\sqcup\s^d$ for each $\s\in T$, and a complete set of representatives $\left\{ \tilde{\bf k}(i)\right\}_{i=1}^{r_\s}$ of $\Z^{\s}/\Z {}^tA_\sigma$. Then, a set of functions  
$\Phi_T:=\displaystyle\bigcup_{\sigma\in T}\left\{ \psi_{\s^d,\tilde{\bf k}(i)}^{\s^u}(z;c)\right\}_{i=1}^{r_\s}$ 
is a set of linearly independent \change{(over the field of translation invariant meromorphic functions in $c$)} holomorphic solutions of $GG(A)$ on $U_{T}:=\displaystyle\bigcap_{\sigma\in T}U_\sigma\neq\varnothing$ where $r_\s$ is the cardinality of the group $\Z^{n}/\Z A_\sigma.$ 

\end{prop}

\begin{proof}
For any ${\bf k}\in\Z^{n}$, we put $\Lambda_{\bf k}=\{ {\bf k+m}\in\Z^{\bs}\mid A_{\bs}{\bf m}\in\Z A_\s\}$ and

\begin{align}
\varphi_{\s^d,{\bf k}}^{\s^u}(z;c)=&
\sum_{{\bf k+m}\in\Lambda_{\bf k}}
\frac{
\prod_{i\in\s^u}\Gamma(p_{\s i}(c+A_{\bs}({\bf k+m})))
}{
\prod_{i\in\s^d}\Gamma(1-p_{\s i}(c+A_{\bar{\sigma}}({\bf k+m})))({\bf  k+m})!
}
\times\nonumber\\
&\prod_{i\in\s^u}(e^{\pi\ii}z_i)^{-p_{\s i}(c+A_{\bar{\sigma}}({\bf k+m}))}
\prod_{i\in\s^d}z_i^{-p_{\s i}(c+A_{\bar{\sigma}}({\bf k+m}))}
z_{\bar{\sigma}}^{\bf k+m}.
\end{align}

\noindent
For any column vector $\tilde{\bf k}\in\Z^{\s}$, we easily see that the identity
\begin{equation}
\psi_{\s^d,\tilde{\bf k}}^{\s^u}(z;c)=e^{-2\pi\ii{}^t\tilde{\bf k}A_\s^{-1}c}\sum_{j=1}^{r_\s}e^{-2\pi\ii{}^t\tilde{\bf k}A_\s^{-1}A_{\bs}{\bf k}(j)}\varphi_{\s^d,{\bf k}(j)}^{\s^u}(z;c)
\end{equation}
holds. Here, ${}^t\tilde{\bf k}$ denotes the transpose of the column vector $\tilde{\bf k}$. Thus, we have
\begin{equation}\label{eqn:PhiPsi}
\begin{pmatrix}
\psi_{\s^d,\tilde{\bf k}(1)}^{\s^u}(z;c)\\
\vdots\\
\psi_{\s^d,\tilde{\bf k}(r_\s)}^{\s^u}(z;c)
\end{pmatrix}
=
\diag\left( e^{-2\pi\ii{}^t\tilde{\bf k}(i)A_\s^{-1}c}\right)_{i=1}^{r_\s}
\left( e^{-2\pi\ii{}^t\tilde{\bf k}(i)A_\s^{-1}A_{\bs}{\bf k}(j)}\right)_{i,j}
\begin{pmatrix}
\varphi_{\s^d,{\bf k}(1)}^{\s^u}(z;c)\\
\vdots\\
\varphi_{\s^d,{\bf k}(r_\s)}^{\s^u}(z;c)
\end{pmatrix}.
\end{equation}
Since it can be readily seen that the set $\displaystyle\bigcup_{\sigma\in T}\left\{ \varphi_{\s^d,{\bf k}(i)}^{\s^u}(z;c)\right\}_{i=1}^{r_\s}$ is linearly independent as in \S3 of \cite{FernandezFernandez} and that $\frac{1}{\sqrt{r_\s}}\left( e^{-2\pi\ii{}^t\tilde{\bf k}(i)A_\s^{-1}A_{\bs}{\bf k}(j)}\right)_{i,j}
$ is a unitary matrix, (\ref{eqn:PhiPsi}) shows the proposition.
\end{proof}

\begin{rem}
We set $\varphi_{\s,{\bf k}}(z;c):=\varphi_{\s,{\bf k}}^\varnothing (z;c)$. Using the reflection formula of Gamma function, it is straightforward to write down any $\varphi_{\s^d,{\bf k}}^{\s^u}(z;c)$ as a linear combination of $\varphi_{\s,{\bf k}}(z;c)$. The series $\varphi_{\s,{\bf k}}(z;c)$ is called a $\Gamma$-series (\cite[ \S1]{GKZToral}). We also call the series $\psi^{\s^u}_{\s^d,\tilde{\bf k}}(z;c)$ a $\Gamma$-series.
\end{rem}

\change{
Let us establish an isomorphism between a homology group and the solution space of GG system. First of all, we need to specify the function space on which we take solutions of GG system. Let $V$ be an subset of $\C^N$.
We write $f(z;c)\in\OM(V\times\C^n)$ if there exist linear forms $L_i$ with integral coefficients, complex numbers $\lambda_j$, integers $m_{i,j}$ and $h\in\C[c]$ such that $g(z;c):=h(c)\prod_{i,j}(\exp(2\pi\ii L_i(c))-\lambda_j)^{m_{i,j}}f(z;c)$ is holomorphic on $V\times\C^n$ and for any compact subset $K\subset V$ and for any pair of real numbers $r_i\leq R_i$ ($i=1,\dots,n$), there are constants $C>0$, $a>0$ and $R>0$ so that on the set $K\times \{ c\in\C^n\mid r_i\leq {\rm Re}\; c_i\leq R_i,\ |{\rm Im} c_i|\geq R\}$, one has an inequality $|g(z;c)|\leq C\exp\{ a\sum_{i=1}^n|{\rm Im} c_i|\}$.
Similarly, we write $\mathscr{M}(\C^n)$ for the set of functions $f(c)$ on $\C^n$ such that $pr^*f(z;c)\in\OM(V\times\C^n)$ where $pr:V\times\C^n\rightarrow\C^n$ is the projection.
For any pair of open subsets $V_1\subset V_2$, we can define a natural restriction morphism $\OM(V_2\times\C^n)\rightarrow \OM(V_1\times\C^n)$. The symbol $\OM_{z}$ denotes the inductive limit $\underset{z\in V}{\varinjlim}\OM(V\times\C^n)$.
We write $\C[T]_{loc}$ for the localization of $\C[T_1,\dots,T_n]$ obtained by localizing the elements of the form $T^L-\lambda$ where $L\in\Z^n$ and $\lambda\in\C$.
It is easy to see that the set $\{ f(c)\in\mathscr{M}(\C^n)\mid f(c+{\bf a}(j))=f(c)\text{ for any }j=1,\dots,N\}$ is isomorphic to $\C[T]_{loc}$ through the correspondence $T_i=e^{2\pi\ii c_i}$.
The fraction field of $\C[T_1,\dots,T_n]$ is denoted by $\C(T)$.  
We write $R_A$ for the ring of difference-differential operators $\C\langle z_1,\dots,z_N,\partial_1,\dots,\pa_N,c_1,\dots,c_n,\change{\tau}_1^\pm,\dots,\change{\tau}_n^\pm\rangle$ with relations $\change{\tau}_i c_i=(c_i+1)\change{\tau}_i$, $\change{\tau}_i^{-1}c_i=(c_i-1)\change{\tau}_i^{-1}$ and $\pa_iz_i=z_i\pa_i+1$ and other types of product of two generators commute. If $I_{A}$ denotes the left ideal of $R_A$ generated by $E_i$ ($i=1,\dots,n$) and $\pa_j-\change{\tau}^{{\bf a}(j)}$ ($j=1,\dots,N$), we set $GG(A):=R_A/I_A$.
$\OM_{z}$ admits a natural structure of a left $R_A$-module and $\Hom_{R_A}(GG(A),\OM_{z})$ is a module over $\C[T]_{loc}$ through the correspondence $T_i=e^{2\pi\ii c_i}$. We can define the $\C(T)$-vector space $\sol_{GG(A),z}$ of local solutions at $z$ by $\sol_{GG(A),z}:=\C(T)\otimes_{\C[T]_{loc}}\Hom_{R_A}(GG(A),\OM_{z})$.
}
\change{
We write $\Delta_A$ for the convex hull of the column vectors of $A$ and the origin. We write ${\rm Sing}(A)$ for the zero set of principal $A$-determinant, which is defined as a product of $A_{\Gamma}$-discriminants $D_{A_\Gamma}$ for any face $\Gamma$ of $\Delta_A$ which does not contain the origin (\cite[Chapter 9]{GKZbook}).
If $z\notin {\rm Sing}(A)$, any $f(z;c)\in\sol_{GG(A),z}$ is a solution of $M_A(c)$ for any generic $c\in\C^n$. Since the singular locus of $M_A(c)$ is contained in ${\rm Sing}(A)$ by \cite[Theorem 2.14]{SchulzeWaltherIrregularity}\footnote{In \cite{SchulzeWaltherIrregularity}, the configuration matrix $A$ is assumed to be pointed, i.e., it is assumed that there is a linear functional $\phi\in(\Z^n)^\vee$ so that $\phi({\bf a}(j))>0$ for any column vector ${\bf a}(j)$ of $A$.
However, a careful reading of \cite{SchulzeWaltherIrregularity} shows that when the weight vector $L$ of \cite{SchulzeWaltherIrregularity} is a positive vector, \cite[Theorem 2.14]{SchulzeWaltherIrregularity} is true without the assumption that $A$ is pointed.
See also \S5 of \cite{FernandezFernandez}.}, any local solution $f\in \sol_{GG(A),z}$ admits an analytic continuation along any path $\gamma$ in $\C^N\setminus {\rm Sing}(A)$.
}
\change{
Let us fix a point $z\in\C^N\setminus {\rm Sing}(A)$. We prove that the solution space $\sol_{GG(A),z}$ is isomorphic to a certain homology group as $\C(T)$-vector spaces and provide a formula of its dimension.
Below, we follow the construction of \cite{LS}. We set $h(x):=h(x;z):=\sum_{j=1}^Nz_jx^{{\bf a}(j)}$ and regard it as a regular function on $U:={\rm Spec}\;\C[x_i^\pm]$.
We write $p:\widetilde{U}\rightarrow U^{an}$ for the universal covering of a torus $U^{an}$.
We set $\mathcal{L}:=p_!\underline{\C}$ where $\underline{\C}$ is the constant sheaf on $\widetilde{U}$ of which a stalk at any point is $\C$.
Each stalk of $\mathcal{L}$ is naturally equipped with the action of the group of deck transformations ${\rm Deck}(\widetilde{U}/U^{an})$.
Let $T_i$ denotes the element of ${\rm Deck}(\widetilde{U}/U^{an})$ specified by the loop $x_i\mapsto e^{2\pi\ii}x_i$. Since ${\rm Deck}(\widetilde{U}/U^{an})$ is freely generated by $T_1,\dots,T_n$, we regard $\mathcal{L}$ as a $\C[T_1^\pm,\dots,T_n^\pm]$-local system of rank 1.
For any element $c\in\C^n$, we can define the specialization $\mathcal{L}_c:=\C[T_1^\pm,\dots,T_n^\pm]/\mathfrak{m}_c\otimes_{\C[T_1^\pm,\dots,T_n^\pm]}\mathcal{L}$ where $\mathfrak{m}_c$ is the maximal ideal corresponding to a point $(e^{2\pi\ii c_1},\dots,e^{2\pi\ii c_n})\in U$.
Let $X$ be a projective compactification of $U$ so that $h(x)$ extends to a morphism $h:X\rightarrow \PP^1$ and $D=X\setminus U$ is a normal crossing divisor (for a concrete construction of such a compactification, see \cite[ \S4]{EsterovTakeuchi} or \cite[ \S3]{MatsubaraEulerLaplace}).
We decompose $D$ as $D=D_{irr}\cup D_\infty$ where $D_\infty$ is the pole divisor of $h$.
Let $\varpi:\widetilde{X}\rightarrow X$ be the real oriented blow-up along $D$.
Then, a real submanifold $\widetilde{D^{r.d.}}$ of $\widetilde{X}$ is naturally defined as a set of rapid decay directions of the function $e^{h(x)}$ (for the precise definition, see [loc. cit.]).
Let $i:U^{an}\rightarrow U^{an}\cup \widetilde{D^{r.d.}}$ and $j:U\cup \widetilde{D^{r.d.}}\rightarrow\widetilde{X}$ be natural inclusions.
We set $\Homo_n^{r.d.}:=\Homo_n\left( U^{an}\cup \widetilde{D^{r.d.}},\widetilde{D^{r.d.}};\mathcal{L}\right)$ and $\Homo_{n,c}^{r.d.}:=\Homo_n\left( U^{an}\cup \widetilde{D^{r.d.}},\widetilde{D^{r.d.}};\mathcal{L}_c\right)$.
Since $\Homo_n^{r.d.}$ is isomorphic to $\Homo^n\left( \widetilde{X};j_!i_*\mathcal{L}\right)$, $\Homo_n^{r.d.}$ is a finitely generated $\C[T_1^\pm,\dots,T_n^\pm]$-module in view of a basic result of constructible sheaves \cite[10.16 THEOREM]{Borel}.
Thus, if we write $\C[T_1^\pm,\dots,T_n^\pm]_{loc}$ for an appropriate localization of $\C[T_1^\pm,\dots,T_n^\pm]$, $\Homo_n^{r.d.}\otimes_{\C[T_1^\pm,\dots,T_n^\pm]}\C[T_1^\pm,\dots,T_n^\pm]_{loc}$ is a free $\C[T_1^\pm,\dots,T_n^\pm]_{loc}$-module and we see that the natural morphism $\Homo_n^{r.d.}/\mathfrak{m}_c\Homo_n^{r.d.}\rightarrow\Homo_{n,c}^{r.d.}$ is an isomorphism for generic $c\in\C^n$.
If we set $\Homo_{n}^{r.d}(T):=\Homo_n^{r.d.}\otimes_{\C[T_1^\pm,\dots,T_n^\pm]}\C(T)$, the dimension of $\Homo_{n}^{r.d}(T)$ is equal to $\dim_{\C}\Homo_{n,c}^{r.d.}$ for generic $c$.
It is also straightforward to see that neither $\Homo_n^{r.d.}$ nor $\Homo_{n,c}^{r.d.}$ depends on the choice of a compactification $X$. Too see this, let $X^\prime$ be another compactification of $U$ such that $X^\prime\setminus U$ is a normal crossing divisor and $h$ extends to $h^\prime:X^\prime\rightarrow\PP^1$. We can consider a blow-up of the closure of the diagonal embedding $U\subset X\times X^\prime$ to construct a compactification $Y$ of $U$ and morphisms $\pi:Y\rightarrow X$, $\pi^\prime:Y\rightarrow X^\prime$ such that $D_Y:=Y\setminus U$ is a normal crossing divisor, $h$ extends to $h_1:Y\rightarrow \PP^1$, $\pi$ and $\pi^\prime$ are identities on $U$ and $h\circ\pi$ and $h^\prime\circ\pi^\prime$ are identical to the extension $h_1:Y\rightarrow\PP^1$.
We write $\widetilde{Y}$ for the real oriented blow-up of $Y$ along $D_Y$ and define $\widetilde{D_Y^{r.d.}}\subset \widetilde{Y}$ as before. If $i^Y:U^{an}\rightarrow U^{an}\cup \widetilde{D_Y^{r.d.}}$ and $j^Y:U^{an}\cup \widetilde{D_Y^{r.d.}}\rightarrow\widetilde{Y}$ denote natural inclusions, it is enough to prove the identity $\Homo^n\left( \widetilde{X};j_!i_*\mathcal{L}\right)=\Homo^n\left( \widetilde{Y};j^Y_!i^Y_*\mathcal{L}\right)$. In view of the fact $\R i_*\mathcal{L}=i_*\mathcal{L}$, $\R i^Y_*\mathcal{L}=i^Y_*\mathcal{L}$ and the induced morphism $\tilde{\pi}:\widetilde{Y}\rightarrow\widetilde{X}$ induces a proper morphism $\tilde{\pi}:U^{an}\cup \widetilde{D_Y^{r.d.}}\rightarrow U^{an}\cup \widetilde{D^{r.d.}}$, we obtain the desired identity in view of the following commutative diagram:
\begin{equation}
\xymatrix{
\widetilde{Y}\ar[r]^-{\tilde{\pi}}  &\widetilde{X}\\
U^{an}\cup \widetilde{D_Y^{r.d.}}\ar[r]^-{\tilde{\pi}}\ar[u]^-{j^Y} &U^{an}\cup \widetilde{D^{r.d.}}\ar[u]_-{j}\\
U^{an}\ar[r]^-{{\rm id}}\ar[u]^-{i^Y}&U^{an}\ar[u]_-{i}
}.
\end{equation}
\noindent
In the same manner, we can consider the algebraic de Rham cohomology group. Let $\Omega^p_U$ denote the sheaf of $p$-forms on $U$ and set $\Omega_U^p[c]:=\Omega_U^p\otimes_{\C}\C[c]$.
We set $\Homo_{dR}^n:=\mathbb{H}^n(U,(\Omega_U^\bullet[c],\nabla))$ where $\nabla=d_x+d_xh\wedge+\sum_{i=1}^nc_i\frac{dx_i}{x_i}\wedge$.
For a fixed $c\in\C^n$, we set  $\Homo_{dR,c}^n:=\mathbb{H}^n(U,(\Omega_U^\bullet,\nabla))$.
In view of the argument of \cite[ p470-p471]{LS}, the specialization morphism $\Homo^n_{dR}/\mathfrak{n}_c\Homo^n_{dR}\rightarrow\Homo_{dR,c}^n$ is an isomorphism for generic $c\in\C^n$ where $\mathfrak{n}_c$ is the maximal ideal corresponding to a point $c\in\C^n$.
We set $\Homo_{dR}^n(c):=\Homo_{dR}^n\otimes_{\C[c]}\C(c)$. Note that $\Homo_{dR}^n(c)$ carries a natural action of $\tau_i$ and is an example of {\it le complexe d'Aomoto} in the sense of \cite[ p470]{LS}.
Similarly, we set $\Homo_{dR}^n[z]:=\mathbb{H}^n(U,(\Omega_U^\bullet[z,c],\nabla))$ and $\Homo_{dR,c}^n[z]:=\mathbb{H}^n(U,(\Omega_U^\bullet[z],\nabla))$. The cohomology group $\Homo_{dR}^n[z]$ admits a natural action of the ring $R_A$ and especially that of its subalgebra $D_N$ generated by $z_1,\dots,z_N,\partial_1,\dots,\partial_N$ over $\C$.
The action of $\partial_j$ on $[\omega(z)]\in\Homo_{dR}^n[z]$ is given by $\partial_j\bullet[\omega(z)]=[\partial_j\omega(z)+x^{{\bf a}(j)}\omega(z)]$.
$\Homo_{dR,c}^n[z]$ also admits a natural action of $D_N$ with which the specialization morphism $\Homo^n_{dR}[z]/\mathfrak{n}_c\Homo^n_{dR}[z]\rightarrow\Homo_{dR,c}^n[z]$ is a homomorphism of left $D_N$-modules.
Note that $\Homo_{dR,c}^n[z]$ is a holonomic $D_N$-module isomorphic to GKZ system when $c$ is non-resonant. It is generated by the canonical element $[\frac{dx}{x}]$ over $D_N$ where $\frac{dx}{x}:=\frac{dx_1}{x_1}\wedge\cdots\wedge\frac{dx_n}{x_n}$ (\cite[Lemma 4.7]{EsterovTakeuchi}). For any $\overset{\circ}{z}\in\C^N\setminus {\rm Sing}(A)$, we write $\mathfrak{m}_{\overset{\circ}{z}}$ for the maximal ideal of $\C[z]$ corresponding to the point $\overset{\circ}{z}$.
The specialization morphism $\Homo^n_{dR,c}[z]/\mathfrak{m}_z\Homo^n_{dR,c}[z]\rightarrow\Homo^n_{dR,c}$ is an isomorphism in view of \cite[LEMMA 3.3]{Adolphson} and \cite[Corollary 3.8]{SchulzeWaltherLGM}.
}

\change{
For any $[\Gamma]\in\Homo_{n}^{r.d.}$ and $c\in\C^n$, we write $[\Gamma(c)]\in\Homo_{n,c}^{r.d.}$ for its specialization. For any $[\omega]\in\Homo_{dR}^n$, we see that the function $\C^n\ni c\mapsto \int_{\Gamma(c)}e^{h(x)}x^c\omega$ is a holomorphic function and this function is denoted by $\int_\Gamma e^{h(x)}x^c\omega$. For any polynomial $P\in\C[T_1,\dots,T_n]$ and $q(c)\in\C[c]$, we set $\int_{P(T)^{-1}\Gamma} e^{h(x)}x^c\frac{\omega}{q(c)}:=\frac{1}{P(e^{2\pi\ii c_1},\dots,e^{2\pi\ii c_n})q(c)}\int_\Gamma e^{h(x)}x^c\omega$. 
}
\change{
\begin{lem}\label{lem:growth}
Let $D(z;\ve)$ be a disk of radius $\ve>0$ centered at $z$. For any $[\Gamma]\in\Homo_{n}^{r.d.}$ and $[\omega]\in\Homo_{dR}^n(c)$, the period integral $\int_{\Gamma}e^{h(x;z)}x^c\omega$ defines a meromorphic function on $D(z;\ve)\times \C^n$ for a small $\ve>0$ and it belongs to $\OM(D(z;\ve)\times \C^n)$.
\end{lem}
\begin{proof}
The first assertion follows from the beginning of the proof of \cite[Theorem 3.11]{MatsubaraEulerLaplace}. Let us show the second assertion. After a sequence of codimension 2 blow-ups of $X$ if necessary, we may assume that the coordinate function $x_i:U\rightarrow\C$ also extends to a function $x_i:X\rightarrow\PP^1$. At each point of $\widetilde{D^{r.d.}}$, we can find a coordinate system $(r_1,\dots,r_s,\theta_1,\dots,\theta_s;y_{s+1},\dots,y_n)$ of $\widetilde{X}$ on which we have $x^c=\prod_{i=1}^sr_i^{L_i(c)}e^{\ii L_i(c)\theta_i}\prod_{i=1}^n(\varpi^*g_i)^{c_i}$ where $g_i$ is a germ of an invertible holomorphic function on $X$ and $L_i$ is a linear form in $c$ with coefficients in $\Z$. Note that $|r_i^{L_i(c)}|=r_i^{L_i({\rm Re}\; c)}$. In view of this local representation of the integrand, we see that the estimate of the lemma is valid.
\end{proof}
\noindent
The following lemma is a variant of \cite[PROPOSITION 3]{GelfandGraevGG} of which the proof is same as [loc. cit.].
\begin{lemma}\label{lem:IntRep}
For any $[\Gamma]\in\Homo_{n}^{r.d.}(T)$, the integral $\int_\Gamma e^{h(x;z)}x^c\frac{dx}{x}$ belongs to $\sol_{GG(A),z}$.
\end{lemma}
}

\change{
\noindent
Let $\vol_{\R}$ be the Lebesgue measure on $\R^n$ and set $\vol_{\Z}:=\frac{1}{n!}\vol_{\R}$. The perfectness of the period pairing (\cite{HienRDHomology}) combined with \cite[LEMMA 3.3]{Adolphson} and \cite[Corollary 3.8]{SchulzeWaltherLGM} shows that $\dim_{\C}\Homo^{r.d.}_{n,c}=\dim_{\C}\Homo_{dR,c}^n=\vol_{\Z}(\Delta_A)$ for any generic $c$ (to be more precise, for any non-resonant $c$). Therefore, we have $\dim_{\C(T)}\Homo_n^{r.d.}(T)=\dim_{\C(c)}\Homo_{dR}^n(c)=\vol_{\Z}(\Delta_A)$.
}

\change{
\begin{thm}\label{thm:2.3}
For any $z\in\C^N\setminus {\rm Sing}(A)$, the map 
\begin{equation}\label{eqn:Int}
\Homo_{n}^{r.d.}(T)\ni[\Gamma]\mapsto\int_\Gamma e^{h(x;z)}x^c\frac{dx}{x}\in\sol_{GG(A),z}
\end{equation}
is an isomorphism of $\C(T)$-vector spaces. In particular, one has an identity
\begin{equation}
\dim_{\C(T)}\sol_{GG(A),z}=\vol_{\Z}(\Delta_A).
\end{equation}
\end{thm}
\begin{proof}
Let us set $r=\vol_{\Z}(\Delta_A)$ and take a subset $\{ P_i\}_{i=1}^r\subset D_N$ so that $\{[\omega_i]:=P_i\bullet[\frac{dx}{x}]\}$ is a $\C(c)$-basis of $\Homo_{dR}^n(c)$ of which the specialization at a generic $c$ is a $\C$-basis of $\Homo_{dR,c}^n$. We may assume that $P_1=1$.
We take a $\C(T)$-basis $\{[\Gamma_1],\dots,[\Gamma_r]\}$ so that its specialization to a generic point $T_i=e^{2\pi\ii c_i}$ is a $\C$-basis of $\Homo_{n,c}^{r.d.}$.
By the perfectness of the period pairing $\Homo_{n,c}^{r.d.}\otimes\Homo_{dR,c}^n\rightarrow\C$ (\cite{HienRDHomology}), we see that the determinant of the matrix $\Psi:=\left(\int_{\Gamma_j}e^{h(x;z)}x^c\omega_i\right)_{i,j=1}^r$ is not identically zero as a function of $z,c$. Since it is subject to {\it un syst\`eme rationnel holonome d'EDF de dimension 1} in the sense of \cite{LS}, we deduce from Lemma \ref{lem:growth} and \cite[PROPOSITION 4.1.4]{LS} that $\det\Psi^{-1}$ belongs to $\mathscr{M}(\C^n)$ when $z$ is fixed.
We write $\Psi_1$ for the first row vector of $\Psi$.
With these preparations, we prove that the entries of $\Psi_1$ form a basis of $\sol_{GG(A),z}$ over $\C(T)$.
The linear independence is immediate from \cite[Theorem 4.5]{EsterovTakeuchi}.
Let us take any $\varphi\in\sol_{GG(A),z}$. For a generic choice of $c\in\C^n$, $\varphi$ is a solution of $M_A(c)$ and there is a unique column vector $S(c)\in\C^r$ so that $\varphi=\Psi_1S(c)$, from which we deduce an identity $(P_i\varphi)_i=\Psi S(c)$ where we regard $(P_i\varphi)_i$ as a column vector. Since $\varphi(z;c)\in\sol_{GG(A),z}$, we obtain $\Psi_1(z;c+{\bf a}(j))S(c+{\bf a}(j))=\varphi(z;c+{\bf a}(j))=\partial_j\varphi(z;c)=\partial_j\Psi_1(z;c)S(c)=\Psi_1(z;c+{\bf a}(j))S(c)$.
Thus, we see that $S(c+{\bf a}(j))=S(c)$.
Therefore, the identity $S(c)=\Psi^{-1}(P_i\varphi)_i$ combined with the fact that $\det\Psi^{-1}$ belongs to $\mathscr{M}(\C^n)$ when $z$ is fixed proves that $S(c)$ belongs to $\C(T)$. The argument above shows that the entries of $\Psi_1$ span $\sol_{GG(A),z}$ over $\C(T)$. 
\end{proof}
}

\change{
\noindent
The explicit construction for the integration contour in \cite[ \S6]{MatsubaraEulerLaplace} shows that each element of $\Phi_T$ belongs to $\sol_{GG(A),z}$. Therefore, we obtain a
\begin{cor}
Fix a convergent regular triangulation $T$. Then the set $\Phi_T$ of Proposition \ref{prop:2.2} is a $\C(T)$-basis of $\sol_{GG(A),z}$ for any $z\in U_T$.
\end{cor}
}

\change{
In the discussion above, we fixed a point $z\in\C^N\setminus {\rm Sing}(A)$. If we write $\Homo_{n,z}^{r.d.}(T)$ for $\Homo_n^{r.d.}(T)$, it is straightforward to prove that there is a local system $\mathcal{H}^{r.d.}_{n}(T)=\displaystyle\bigcup_{z\in\C^N\setminus {\rm Sing}(A)}\Homo^{r.d.}_{n,z}(T)\rightarrow\C^N\setminus {\rm Sing}(A)$ of $\C(T)$-vector spaces. Correspondingly, we can also construct a sheaf of solutions of GG system. Indeed, it is easy to see that $V\mapsto\OM(V\times\C^n)$ is a sheaf and so is $\sol_{GG(A)}:=\left(V\mapsto\C(T)\otimes_{\C[T]_{loc}}\Hom_{R_A}(GG(A),\OM(V\times\C^n))\right)$. We can also construct a morphism $\mathcal{H}^{r.d.}_{n}(T)\rightarrow\sol_{GG(A)}$ whose stalk is identical to (\ref{eqn:Int}). See the proofs of \cite[Proposition 3.4. and Theorem 3.5.]{HienRoucairol} which can be adapted to our setting. This observation shows that, for any $z_1,z_2\in\C^N\setminus {\rm Sing}(A)$, we have an isomorphism of solution spaces $\sol_{GG(A),z_1}\simeq\sol_{GG(A),z_2}$ given by an analytic continuation along a path connecting $z_1$ and $z_2$. We summarize the discussion above as a
}

\change{
\begin{prop}
A monodromy matrix or a connection matrix of GG system belongs to $\C(T)^{r\times r}$ with $r=\vol_{\Z}(\Delta_A).$
\end{prop}
}

\change{
\begin{rem}
Theorem \ref{thm:2.3} can be modified to give a formula of the dimension of the solution space of GG system when the column vectors of $A$ do not generate the lattice $\Z^n$. Suppose that the $\Q$-span of the column vectors of $A$ generates $\Q^n$. Let $Q\in\Z^{n\times n}$ be a matrix such that $\Z A=\Z Q$. We put $A^\prime=Q^{-1}A\in\Z^{n\times N}$ and $c^\prime=Q^{-1}c={}^t(c_1^\prime,\dots,c^\prime_n)$. By definition, we have $\Z A^{\prime}=\Z^{n}$. It can readily be seen that $M_A(c)=M_{A^\prime}(c^\prime)$ and $GG(A)=GG(A^\prime)$. Replacing $A$ and $c$ by $A^\prime$ and $c^\prime$, the argument of Theorem \ref{thm:2.3} gives a formula 
\begin{equation}\label{eqn:OT}
\dim_{\C(T)}\sol_{GG(A),z}=\frac{\vol_{\Z}(\Delta_A)}{[\Z^{n}:\Z A]}=:{\rm rank}GG(A),
\end{equation}
where the action of $T_i$ on $\sol_{GG(A),z}$ is given by $T_i=e^{2\pi\ii c_i^\prime}$. The formula (\ref{eqn:OT}) is an analytic counterpart of the result of \cite{OharaTakayama}.
\end{rem}
\begin{rem}
One can develop an isomorphism between a homology group associated to other types of integral representations and the solution space of GG system along the line of the discussion in this subsection. See \cite{MatsubaraEulerLaplace}.
\end{rem}
}

\subsection{Unique solvability of the boundary value problem}\label{subsec:2.2}

A simple, but important observation is that GG system has a structure of boundary value problem. Let ${\bf a}(N+1)\in\Z^{n}$ be a lattice vector and put $\tilde{A}=(A|{\bf a}(N+1))$. We consider a formal solution $f(z,z_{N+1};c)$ of $\GG(\tilde{A})$. Then, we can easily see that its boundary value ${\rm bv}_{N+1}(f)(z;c):=f(z,0;c)$ along $\{ z_{N+1}=0\}$ (if it exists) gives rise to a solution of $\GG(A)$. Thus, the boundary value map ${\rm bv}_{N+1}$ takes the space of formal solutions of  $GG(\tilde{A})$ to that of $GG(A)$. It is now natural to ask if this boundary value map gives rise to a bijection.

Let us make this observation more precise. In this section, we do not assume that the column vectors of $A$ generate the lattice $\Z^n$ but we assume that they generate $\Q^n$ over $\Q$. We consider a left $R_A$-module $M$. We set $M[[z_{N+1}]]:=M\otimes_{\C}\C[[z_{N+1}]]$. The natural action of $\pa_{N+1}$ onto $\C[[z_{N+1}]]$ induces an action of $R_{\tilde{A}}$ onto $M[[z_{N+1}]]$. Any $R_{\tilde{A}}$-morphism $F:GG(\tilde{A})\rightarrow M[[z_{N+1}]]$ is determined by the value of $[1]\in GG(\tilde{A})$. By abuse of notation, we write $F(z,z_{N+1};c)=\sum_{m=0}^\infty \frac{f_m(z;c)}{m!}z_{N+1}^m$ for the morphism $F$ where $f_m(z;c)\in M$. We can naturally associate the boundary value ${\rm bv}_{N+1}(F(z,z_{N+1};c))$ to the hyperplane $\{ z_{N+1}=0\}$ by setting ${\rm bv}_{N+1}(F(z,z_{N+1};c)):=f_0(z;c)\in M$. The symbol $F(z,0;c)$ also denotes the boundary value ${\rm bv}_{N+1}(F(z,z_{N+1};c))=f_0(z;c)$. It can readily be seen that there is a unique $R_A$-morphism from $GG(A)$ to $M$ which sends $[1]$ to $f_0(z;c)\in M$. This $R_A$-morphism is also denoted by the symbols $F(z,0;c)$, ${\rm bv}_{N+1}(F(z,z_{N+1};c))$ or $f_0(z;c)$. Lastly, for any $R_A$-morphism $f:GG(A)\rightarrow M$ and a vector ${\bf v}\in\Z^n$, we write $f(z;c+{\bf v})$ for the image of $[\change{\tau}^{\bf v}]$.

\begin{thm}\label{prop:BVP}
The map ${\rm bv}_{N+1}$ induces a linear isomorphism
\begin{equation}
\begin{array}{cccc}
{\rm bv}_{N+1}:&\Hom_{R_{\tilde{A}}}(GG(\tilde{A}),M[[z_{N+1}]])&\tilde{\longrightarrow}&\Hom_{R_{A}}(GG(A),M)\\
&\rotatebox{90}{$\in$}& &\rotatebox{90}{$\in$}\\
&F(z,z_{N+1};c) &\mapsto &F(z,0;c),
\end{array}
\end{equation}
whose inverse is given by the formula
\begin{equation}\label{GreenKernel}
\mathfrak{D}_{N+1}:f(z;c)\mapsto \sum_{m=0}^\infty f(z;c+m{\bf a}(N+1))\frac{z_{N+1}^m}{m!}
\end{equation}
for any $f(z;c)\in\Hom_{R_{A}}(GG(A),M)$.
\end{thm}

\begin{proof}
For an element $f(z;c)\in\Hom_{R_{A}}(GG(A),M)$, we set $F(z,z_{N+1};c):=\sum_{m=0}^\infty f(z;c+m{\bf a}(N+1))\frac{z_{N+1}^m}{m!}$. First, we show that $F(z,z_{N+1};c)$ is a solution of $GG(\tilde{A})$, i.e., $F$ is a $R_{\tilde{A}}$-morphism. In view of the fact that $\mathfrak{D}_{N+1}\circ\pa_j=\pa_j\circ\mathfrak{D}_{N+1}$ for any $j=1,\dots,N$, we have $\partial_jF(z,z_{N+1};c)=F(z,z_{N+1};c+{\bf a}(j)).$ On the other hand, we have 
\begin{align}
\pa_{N+1}F(z,z_{N+1};c)&=\sum_{n=1}^\infty \frac{f(z;c+n{\bf a}(N+1))}{(n-1)!}z_{N+1}^{n-1}\\
&=\sum_{n=0}^\infty \frac{f(z;c+{\bf a}(N+1)+n{\bf a}(N+1))}{n!}z_{N+1}^n\\
&=F(z,z_{N+1};c+{\bf a}(N+1)).
\end{align}

\noindent
As for Euler equations, for any $i=1,\dots,n$, we have
\begin{align}
 &\left[ \sum_{j=1}^Na_{ij}\theta_j+a_{iN+1}\theta_{N+1}+c_i\right]\left( \sum_{n=0}^\infty \frac{f(z;c+n{\bf a}(N+1))}{n!}z_{N+1}^n\right)\\
=&\sum_{n=0}^\infty \frac{z_{N+1}^n }{n!}\{ -(c_i+na_{iN+1})+na_{iN+1}+c_i\}f(z;c+n{\bf a}(N+1))\\
=&0.
\end{align}
This shows that $F(z,z_{N+1};c)$ is a solution of $GG(\tilde{A})$ and that $\mathfrak{D}_{N+1}:\Hom_{R_{A}}(GG(A),M)\rightarrow\Hom_{R_{\tilde{A}}}(GG(\tilde{A}),M[[z_{N+1}]])$ is well-defined. It is easy to see that ${\rm bv}_{N+1}\circ\mathfrak{D}_{N+1}=id.$ 

Next, we suppose $F(z,z_{N+1};c)\in\Hom_{R_{\tilde{A}}}(GG(\tilde{A}),M[[z_{N+1}]])$. We expand it as $F(z,z_{N+1};c)=\sum_{m=0}^\infty\frac{1}{m!}\left\{\left(\pa_{N+1}\right)^mF\right\}(z,0;c)z_{N+1}^m.$ Since $\left(\pa_{N+1}\right)^mF(z,z_{N+1};c)=F(z,z_{N+1};c+m{\bf a}(N+1))$, if we put $f(z;c)={\rm bv}_{N+1}(F(z,z_{N+1};c))$, we obtain the identity $F(z,z_{N+1};c)=\mathfrak{D}_{N+1} f(z;c)$. This argument shows the identity $\mathfrak{D}_{N+1}\circ {\rm bv}_{N+1}=id.$
\end{proof}

\begin{rem}
Let $\s\subset\{ 1,\dots,N\}$ be a simplex and let $\s=\s^u\sqcup\s^d$ be  a partition of $\s$. It is easily verified that we have an equality ${\rm bv}_{N+1}(\psi^{\s^u}_{\s^d}(z,z_{N+1};c))=\psi^{\s^u}_{\s^d}(z;c)$ and therefore, $\mathfrak{D}_{N+1}\psi^{\s^u}_{\s^d}(z;c)=\psi^{\s^u}_{\s^d}(z,z_{N+1};c)$.
\end{rem}

\begin{rem}
If $M$ is a left $R_A\otimes_{\C}\C[T]_{loc}$-module, ${\rm bv}_{N+1}$ is an isomorphism of $\C[T]_{loc}$-modules.
\end{rem}

\subsection{Monodromy invariant subspaces}\label{subsec:2.3}

In this subsection, we assume that the column vectors of $A$ generate the lattice $\Z^n$. \change{We prove a decomposition of $\sol_{GG(A),z}$ into monodromy invariant subspaces. Note that the decomposition in Theorem \ref{thm:2.7} is a trivial decomposition when $A$ is homogeneous.} We take a facet $F$ of $\Delta_A$ which does not contain the origin. \change{We often identify $F$ with the convex hull of column vectors ${\bf a}(j)$ lying on $F$ and the origin.} We fix $n\times n$ integer matrix $Q$ such that $\Z A_F=\Z Q$. For any local solution $f(z_F;c)\in\sol_{GG(A_F),z}$ and $\tilde{\bf k}\in\Z^{n}/\Z{}^tQ$, we have $e^{2\pi\ii{}^t\tilde{\bf k}Q^{-1}c}f(z_F;c)\in\sol_{GG(A_F),z}$. Therefore, the function
$\prod_{j\notin F}\mathfrak{D}_j \left( e^{2\pi\ii{}^t\tilde{\bf k}Q^{-1}c}f(z_F;c)\right)$ is a formal solution of $GG(A)$. Let us take a fundamental basis of solutions $\Phi_{A_F}(z_F;c)$ of $GG(A_F)$ and put $m=\rank GG(A_F)$ and a complete system of representatives $\{{\bf k}(i)\}_{i=1}^r$ of $\Z^n/\Z {}^tQ$. Regarding $\Phi_{A_F}$ as a row vector, we put 
\begin{equation}
\Psi_{A_F}(z_F;c)=\left( e^{2\pi\ii{}^t\tilde{\bf k}(1)Q^{-1}c}\Phi_{A_F}(z_F;c),\dots,e^{2\pi\ii{}^t\tilde{\bf k}(r)Q^{-1}c}\Phi_{A_F}(z_F;c)\right).
\end{equation}
We consider a path $\gamma$ of analytic continuation along which one has a relation $\gamma_*\Phi_{A_F}(z_F;c)=\Phi_{A_F}(z_F;c)M(c)$. Then, one has a relation $\gamma_*\Psi_{A_F}(z_F;c)=\Psi_{A_F}(z_F;c) \begin{pmatrix} M(c)& & \\ &\ddots& \\ & &M(c)\end{pmatrix}.$ Now we put $\Psi_F(z;c):=\prod_{j\notin F}\mathfrak{D}_j\Psi_{A_F}(z_F;c)$. We define the tensor product $A\otimes B$ of $m_1\times m_1$ matrix $A=(a_{ij})$ and $m_2\times m_2$ matrix $B$ by the formula 
\begin{equation}
A\otimes B=
\begin{pmatrix}
a_{11}B&\cdots&a_{1m_1}B\\
\vdots&\ddots&\vdots\\
a_{m_11}B&\cdots&a_{m_1m_1}B
\end{pmatrix}.
\end{equation}
Note that $A\otimes B$ is invertible if and only if both $A$ and $B$ are invertible and the inverse in this case is given by $A^{-1}\otimes B^{-1}$
Then,
\begin{align}
\Psi_F(z;c)&=(\Psi_1,\dots,\Psi_r)\\
&=\left( F_1(z;c),\dots,F_r(z;c)\right) {}^tC\otimes I_m.
\end{align}
Here, we have put 
\begin{equation}
F_l(z;c)=\sum_{[A_F{\bf m}]=[A_{\bar{F}}{\bf k}(l)]}\Phi_{A_F}(z_F,c+A_{\bar{F}}{\bf m})\frac{z_{\bar{F}}^{\bf m}}{{\bf m}!}
\end{equation}
and
\begin{equation}
C=\diag\left( e^{2\pi\ii{}^t\tilde{\bf k}(i)Q^{-1}c}\right)_{i=1}^r\left( e^{2\pi\ii{}^t\tilde{\bf k}(i)Q^{-1}A_{\bar{F}}{\bf k}(j)}\right)_{i,j=1}^r,
\end{equation}
where the symbol $[\  ]$ denotes the equivalence class in the group $\Z^n/\Z A_F=\Z^n/\Z Q$. Note that we have 
\begin{equation}
{}^tC^{-1}
=\frac{1}{r}\diag\left( e^{-2\pi\ii{}^t\tilde{\bf k}(i)Q^{-1}c}\right)_{i=1}^r\left( e^{-2\pi\ii{}^t\tilde{\bf k}(i)Q^{-1}A_{\bar{F}}{\bf k}(j)}\right)_{i,j=1}^r.
\end{equation}

\noindent
Putting $M_l=M(c+A_{\bar{F}}{\bf k}(l))$ and ignoring the problem of convergence for the moment, we have
\begin{align}
\gamma_*\Psi_F&=\prod_{j\notin F}\mathfrak{D}_j\gamma_*\Psi_{A_F}\\
&=\left(
\sum_{l=1}^re^{2\pi\ii{}^t\tilde{\bf k}(1)Q^{-1}(c+A_{\bar{F}}{\bf k}(l))}F_lM_l,\dots,
\sum_{l=1}^re^{2\pi\ii{}^t\tilde{\bf k}(r)Q^{-1}(c+A_{\bar{F}}{\bf k}(l))}F_lM_l
\right)\\
&=\frac{1}{r}\Psi_F({}^tC^{-1}\otimes I_m)
\begin{pmatrix}
C_{11}M_1&\cdots&C_{r1}M_1\\
\vdots&\ddots&\vdots\\
C_{1r}M_r&\cdots&C_{rr}M_r
\end{pmatrix}\\
&=\frac{1}{r}\Psi_F
\begin{pmatrix}
\sum_{l=1}^rC_{1l}^{-1}C_{1l}M_l&\cdots&\sum_{l=1}^rC_{1l}^{-1}C_{rl}M_l\\
\vdots&\ddots&\vdots\\
\sum_{l=1}^rC_{rl}^{-1}C_{1l}M_l&\cdots&\sum_{l=1}^rC_{rl}^{-1}C_{rl}M_l
\end{pmatrix}.
\label{eqn:MIS}
\end{align}
Note that $C^{-1}_{ij}$ stands for the fraction $\frac{1}{C_{ij}}$. If we showed that each entry of $\Psi_F$ is convergent for any $z_{\bar{F}}\in\C^{\bar{F}}$ as a power series, (\ref{eqn:MIS}) shows that the entries of $\Psi_F$ span a monodromy invariant subspace of $\sol_{GG(A),z}$.

\change{
Let us introduce notation necessary to formulate the theorem. As in \S\ref{subsec:2.1}, we write $\C(T)$ for the fraction field of the ring $\C[T_1^\pm,\dots,T^\pm_n]=\C[\Z^n]$. We consider an extension $\Z^n\subset \Z Q^{-1}$ of free abelian groups and we write $\C(T^\prime)$ for the fraction field of the ring $\C[\Z Q^{-1}]=\C[T_1^{\prime\pm},\dots,T_n^{\prime\pm}]$. The extension degree $[\C(T^\prime):\C(T)]$ is equal to the index $[\Z^n:\Z A_F]$. The action of $\C(T^\prime)$ on $\sol_{GG(A_F),z_F}$ is defined by $T^\prime_i=e^{2\pi\ii c^\prime_i}$ and (\ref{eqn:OT}) shows that its dimension over $\C(T^\prime)$ is $\frac{\vol_{\Z}(F)}{[\Z^n:\Z A_F]}$. Therefore, if we regard $\sol_{GG(A_F),z_F}$ as a $\C(T)$-vector space, its dimension is given by $\frac{\vol_{\Z}(F)}{[\Z^n:\Z A_F]}\cdot [\C(T^\prime):\C(T)]=\vol_{\Z}(F)$. On the other hand, the action of $\C(T)$ on $\sol_{GG(A),z}$ is given by $T_i=e^{2\pi\ii c_i}$.
}

\begin{thm}\label{thm:2.7}
Let $z\in\C^N\setminus{\rm Sing}(A)$ be a point. We have a decomposition of the solution space of the GG system into monodromy invariant subspaces
\begin{equation}
\sol_{GG(A),z}=\bigoplus_{\substack{0\notin F<\Delta_A\\ F:\text{ facet}}}S_F.
\end{equation}
Here, $S_F$ is a subspace of $\sol_{GG(A),z}$ canonically isomorphic to $\sol_{GG(A_F),z_F}$ as $\C(T)$-vector spaces through the boundary value map $\prod_{j\notin F}{\rm bv}_{j}$.
\end{thm}

\begin{proof}
We only need to ensure that $\Psi_F$ is well-defined, convergent in $z_{\bar{F}}\in\C^{\bar{F}}$ and $\sol_{GG(A),z}$ is spanned by these functions. We first take a row vector $l_1\in\Z^{1\times n}$ such that $l_1(Q^{-1}A_F)=(1,\dots,1)$. We prolong $l_1$ to a basis $\{ l_j\}_{j=1}^n$ of $\Z^{1\times n}$ and put 
$
L=
\begin{pmatrix}
l_1\\
\hline
\vdots\\
\hline
l_n
\end{pmatrix}
$. Then, we put 
$
A^\prime_F=LQ^{-1}A_F=
\begin{pmatrix}
1& \cdots&1\\
 &{\Huge \tilde{A}_F} &
\end{pmatrix}
$, 
$
A^\prime
=
LQ^{-1}A
$, and 
$
c^\prime=LQ^{-1}c=
\begin{pmatrix}
\gamma\\
\tilde{c}
\end{pmatrix}
$. We show that for any solution $f(z_F,c)$ of $GG(A_F)$, the function $\prod_{j\in\bar{F}}\mathfrak{D}_jf(z_F;c)$ is convergent in $z_{\bar{F}}\in\C^{\bar{F}}$. We write $\tilde{\bf a}(j)$ for the $j$($\in F$)-th column vector of the matrix $\tilde{A}_F$. \change{By construction, $A^\prime_F$ is homogeneous.} By the general theory of Euler integral representation (\cite[Theorem 2.14]{GKZEuler}), any solution of the GKZ system $M_{A^\prime_F(c^\prime)}$ with generic (i.e., non-resonant) $c^\prime$ has the form
\begin{equation}\label{eqn:EI}
f(z_F;c)=e^{\pi\ii\gamma}\Gamma(\gamma)\int_C h_{F}(x;z)^{-\gamma}x^{\tilde{c}}\frac{dx}{x},
\end{equation}
where $h_F(x;z)=\sum_{j\in F}z_jx^{\tilde{\bf a}(j)}$, $x=(x_2,\dots,x_n)$ is a variable on $(n-1)$-dimensional complex torus, $\frac{dx}{x}=\frac{dx_2}{x_2}\wedge\cdots\wedge\frac{dx_n}{x_n}$ and $C$ is a cycle in a twisted homology group.\footnote{To be more precise, $C$ should be taken from a homology group of $U:=\{x\in(\C^*)^{n-1}\mid h_F(x;z)\neq 0\}$ with coefficients in a $\C[T_1,\dots,T_n]$-local system of rank 1. Here, $T_1$ corresponds to a loop around $\{ x\mid h_F(x;z)=0\}$ in $U$ and other $T_i$ correspond to a loop around $\{ x_i=0\}$.} Note that (\ref{eqn:EI}) is also a solution of $GG(A_F)$. Based on this formula, we have a relation
\begin{align}
&\prod_{j\notin F}\mathfrak{D}_jf(z_F;c)\nonumber\\
=&
\sum_{{\bf m}\in\Z_{\geq 0}^{\bar{F}}}
e^{\pi\ii\left(\gamma+\sum_{j\in\bar{F}}a^\prime_{1j}m_j\right)}\Gamma\left( \gamma+\sum_{j\in\bar{F}}a^\prime_{1j}m_j\right)\frac{z_{\bar{F}}^{\bf m}}{{\bf m}!}
\left(
\int_C h_{\tilde{A}_F,z_F}(x)^{-\gamma-\sum_{j\in\bar{F}}a^\prime_{1j}m_j}x^{\tilde{c}+\tilde{A}_{\bar{F}}{\bf m}}\frac{dx}{x}
\right),
\end{align}
where we write $a^\prime_{ij}$ for the $(i,j)$-entry of the matrix $A^\prime$. We set $\bar{F}_+=\{ j\notin F\mid a^\prime_{1j}\geq0\}$ (resp. $\bar{F}_-=\{ j\notin F\mid a^\prime_{1j}<0\}$). By the definition of beta function, we have
\begin{align}
\frac{
\Gamma\left( \gamma+\sum_{j\in\bar{F}}a^\prime_{1j}m_j\right)
}
{
\Gamma\left( \gamma+\sum_{j\in\bar{F}_-}a^\prime_{1j}m_j\right)
\Gamma\left( \sum_{j\in\bar{F}_+}a^\prime_{1j}m_j\right)
}
=&
\frac{1}{\left( 1-e^{-2\pi\ii(\gamma+\sum_{j\in\bar{F}_-}a^\prime_{1j}m_j)}\right)}\nonumber\\
&\int_{C^\prime}t^{\gamma+\sum_{j\in\bar{F}_-}a^\prime_{1j}m_j-1}(1-t)^{\sum_{j\in\bar{F}_+}a^\prime_{1j}m_j-1}dt,
\end{align}
where the contour $C^\prime$ begins from $t=1$, approaches $t=0$, turns around the origin in the negative direction, and goes back to $t=1$. We easily see that the inequality
\begin{equation}
\left|\Gamma\left( \gamma+\sum_{j\in\bar{F}}a^\prime_{1j}m_j\right)
\right|\leq C_1^{|{\bf m}|}\Gamma\left( \sum_{j\in\bar{F}_+}a^\prime_{1j}m_j\right)
\end{equation}
holds for some $C_1>0$. Let us observe that $a_{1j}^\prime <1$ for any $j\in\bar{F}$ since $F$ is a facet of $\Delta_A$. Taking into account that the contour $C$ can be taken so that it does not meet the vanishing locus of $h_F(x;z)$, we have an estimate

\begin{align}
&\left|
e^{\pi\ii\left(\gamma+\sum_{j\in\bar{F}}a^\prime_{1j}m_j\right)}\Gamma\left( \gamma+\sum_{j\in\bar{F}}a^\prime_{1j}m_j\right)
\left(
\int_C h_{F}(x;z)^{-\gamma-\sum_{j\in\bar{F}}a^\prime_{1j}m_j}x^{\tilde{c}+\tilde{A}_{\bar{F}}{\bf m}}\frac{dx}{x}
\right)
\right| \nonumber\\
\leq &C_2^{|{\bf m}|}\Gamma\left( \sum_{j\in\bar{F}_+}a^\prime_{1j}m_j\right),\label{eqn:2.41}
\end{align}
which ensures that 
$
\prod_{j\notin F}\mathfrak{D}_jf(z_F;c)
$ is convergent for any $z_j\in\C$ with $j\notin F$.

We claim that a set $\displaystyle\bigcup_{\substack{0\notin F<\Delta_A\\ F:\text{ facet}}}\{ \Psi_{F}(z;c)\}$ is a basis of solutions of $GG(A)$. For this purpose, we take a convergent regular triangulation $T$ of $A_F$. Then, $T$ induces a regular triangulation to each facet $F$ which does not contain the origin, i.e., if the symbol $T_F$ denotes the set $\{\s\in T\mid \s\subset F\}$, $T_F$ is a regular triangulation. We set $A_\s^\prime=Q^{-1}A_\s$ and fix a complete system of representatives $\{[\tilde{\bf k}^\prime(j)]\}_{j=1}^{r_\s}$ of $\Z^\s/\Z{}^tA_\s^\prime$. In view of Proposition \ref{prop:2.2}, we may set $\Phi_{A_F}(z_F;c)=\bigcup_{\s\in T_F}\{ \psi_{\s,\tilde{\bf k}^\prime(j)}(z_F:c)\}_{j=1}^{r_\s}$. By a direct computation, we obtain a relation $\prod_{j\notin F}\mathfrak{D}_j\left( e^{2\pi\ii{}^t\tilde{\bf k}Q^{-1}c}\psi_{\s,\tilde{\bf k}^\prime(j)}(z_F;c)\right)=\psi_{\s,{}^tA_\s^\prime\tilde{\bf k}+\tilde{\bf k}^\prime(j)}(z;c)$. Combining this fact with the exact sequence
\begin{equation}
0\rightarrow\Z^n/\Z{}^tQ\overset{{}^tA_\s^\prime\times}{\rightarrow}\Z^\s/\Z{}^tA_\s\rightarrow\Z^\s/\Z{}^tA_\s^\prime\rightarrow 0,
\end{equation}
we can see that $\displaystyle\bigcup_{\substack{0\notin F<\Delta_A\\ F:\text{ facet}}}\{ \Psi_{F}(z;c)\}$ is identical to the basis consisting of $\Gamma$-series discussed in Proposition \ref{prop:2.2}.

\end{proof}


\change{
\begin{rem}
The estimate (\ref{eqn:2.41}) shows that any element of $S_F$ is entire in variables $z_{\bar{F}}$ and its singular locus is contained in ${\rm Sing}(A_F)\times\C^{\bar{F}}$. Note that for any point $z_{\bar{F}}\in \C^{\bar{F}}$, the inclusion $(\C^F\setminus {\rm Sing}(A_F))\times\{ z_{\bar{F}}\}\hookrightarrow(\C^F\setminus {\rm Sing}(A_F))\times\C^{\bar{F}}$ is a homotopy equivalence.
\end{rem}
}

\noindent
Let us discuss the case when $c$ is fixed. For any $\tilde{l}=1,\dots,r$, and for any column vector $v$, we have
\begin{equation}
\frac{1}{r}
\begin{pmatrix}
\sum_{l=1}^rC_{1l}^{-1}C_{1l}M_l&\cdots&\sum_{l=1}^rC_{1l}^{-1}C_{rl}M_l\\
\vdots&\ddots&\vdots\\
\sum_{l=1}^rC_{rl}^{-1}C_{1l}M_l&\cdots&\sum_{l=1}^rC_{rl}^{-1}C_{rl}M_l
\end{pmatrix}
\begin{pmatrix}
C_{1\tilde{l}}^{-1}v\\
\hline
\vdots\\
\hline
C_{r\tilde{l}}^{-1}v
\end{pmatrix}
=
\begin{pmatrix}
C_{1\tilde{l}}^{-1}M_{\tilde{l}}v\\
\hline
\vdots\\
\hline
C_{r\tilde{l}}^{-1}M_{\tilde{l}}v
\end{pmatrix}.
\end{equation}

\noindent
Here, we have used the formula $\sum_{l^\prime=1}^rC_{l^\prime l}C^{-1}_{l^\prime \tilde{l}}=r\delta_{l\tilde{l}}$. Let us write $\Psi_i$ as $\Psi_i=(\Psi_{i1},\dots,\Psi_{ir})$. The computation above shows that for each $l=1,\dots,r$, the space ${\rm span}_\C\left\{ \sum_{i=1}^rC_{il}^{-1}\Psi_{ij}\right\}_{j=1}^m$ is monodromy invariant and is isomorphic to $\sol_{M_{A_F}(c+{\bf k}(l)),z}.$ Recall that we say that $c$ is very generic with respect to $T$ if for any simplex $\s\in T$ and for any vector ${\bf k}\in\Z^{\s}$, the vector $A_\s^{-1}(c+{\bf k})$ does not have an integral entry. Summarizing the argument above, we have the following theorem.

\begin{cor}\label{cor:FFW}
Suppose $z\notin{\rm Sing}(A)$ and $c$ is very generic with respect to a convergent regular triangulation $T$. Then, one has the following canonical decomposition of the monodromy representation.

\begin{equation}
\sol_{M_A(c),z}=\bigoplus_{\substack{0\notin F<\Delta_A\\ F:\text{ facet}}}S_F,
\end{equation}
where $S_F$ is a subspace of $\sol_{M_{A}(c),z}$ non-canonically isomorphic to $\displaystyle\bigoplus_{{\bf k}\in R}\sol_{M_{A_F}(c+{\bf k})}$. Here, $R$ is a complete system of representatives of $\Z^{n}/\Z A_F$.

\end{cor}

\begin{rem}
Corollary \ref{cor:FFW} is compatible with the description of the derived restriction of GKZ system obtained in \cite{FFW}.
\end{rem}

\begin{exa}
We consider GG system for the matrix $
A=
\begin{pmatrix}
1&0&0&1&1\\
0&1&0&1&0\\
0&0&1&-1&-1
\end{pmatrix}
$. Any solution $f(z;c)$ of $GG(A)$ is of the form

\begin{equation}\label{eqn:2.46}
f(z;c)=z_1^{-c_1}z_2^{-c_2}z_3^{-c_3}F\left(\frac{z_3z_4}{z_1z_2},\frac{z_3z_5}{z_1};\substack{c_1,c_2\\ c_3}\right),
\end{equation}
where $F\left(z,\zeta;\substack{c_1,c_2\\ c_3}\right)$ is subject to a system of difference-differential equations
\begin{align}
(\theta_z+\theta_\zeta+c_1)F\left(z,\zeta;\substack{c_1,c_2\\ c_3}\right)&=-F\left(z,\zeta;\substack{c_1+1,c_2\\ c_3}\right)\label{eqn:Phi11}\\
(\theta_z+c_2)F\left(z,\zeta;\substack{c_1,c_2\\ c_3}\right)&=-F\left(z,\zeta;\substack{c_1,c_2+1\\ c_3}\right)\\
(\theta_z+\theta_\zeta-c_3)F\left(z,\zeta;\substack{c_1,c_2\\ c_3}\right)&=F\left(z,\zeta;\substack{c_1,c_2\\ c_3+1}\right)\\
\partial_zF\left(z,\zeta;\substack{c_1,c_2\\ c_3}\right)&=F\left(z,\zeta;\substack{c_1+1,c_2+1\\ c_3-1}\right)\\
\partial_\zeta F\left(z,\zeta;\substack{c_1,c_2\\ c_3}\right)&=F\left(z,\zeta;\substack{c_1+1,c_2\\ c_3-1}\right)\label{eqn:Phi12}.
\end{align}
Here, we have set $\theta_z=z\frac{\partial}{\partial z}$ and $\theta_\zeta=\zeta\frac{\partial}{\partial \zeta}$. It is easy to see that the Newton polytope $\Delta_A$ has two facets which do not contain the origin (Figure \ref{fig:NP}).  According to Theorem \ref{thm:2.7}, we see that the facet $1234$ corresponds to a $2$-dimensional monodromy invariant subspace. It is straightforward to see that the boundary value $F\left(z;\substack{c_1,c_2\\ c_3}\right):=F\left(z,0;\substack{c_1,c_2\\ c_3}\right)$ is subject to a system of difference-differential equations
\begin{align}
(\theta_z+c_1)F\left(z;\substack{c_1,c_2\\ c_3}\right)&=-F\left(z;\substack{c_1+1,c_2\\ c_3}\right)\label{eqn:2f11}\\
(\theta_z+c_2)F\left(z;\substack{c_1,c_2\\ c_3}\right)&=-F\left(z;\substack{c_1,c_2+1\\ c_3}\right)\\
(\theta_z-c_3)F\left(z;\substack{c_1,c_2\\ c_3}\right)&=F\left(z;\substack{c_1,c_2\\ c_3+1}\right)\\
\partial_zF\left(z;\substack{c_1,c_2\\ c_3}\right)&=F\left(z;\substack{c_1+1,c_2+1\\ c_3-1}\right)\label{eqn:2f12}.
\end{align}
Setting $\alpha:=c_1,\beta:=c_2,\gamma:=1-c_3$ and ${}_2f_1\left(z;\substack{\alpha,\beta\\ \gamma}\right):=\sum_{m=0}^\infty\frac{\Gamma(\alpha+m)\Gamma(\beta+m)}{\Gamma(\gamma+m)m!}z^m$, it is easy to see that the function $e^{-\pi\ii(\alpha+\beta)}{}_2f_1\left(z;\substack{\alpha,\beta\\ \gamma}\right)$ is a solution of the system (\ref{eqn:2f11})-(\ref{eqn:2f12}). Since the function ${}_2f_1\left(z;\substack{\alpha,\beta\\ \gamma}\right)$ is essentially the Gau\ss' hypergeometric function, the analytic continuations of it give rise to a two dimensional space of functions over the field $\C(e^{2\pi\ii\alpha},e^{2\pi\ii\beta},e^{2\pi\ii\gamma})$. Thus, the analytic continuations of the function $e^{-\pi\ii(\alpha+\beta)}\sum_{n=0}^\infty{}_2f_1\left(z;\substack{\alpha+n,\beta+n\\ \gamma+n}\right)\zeta^n$ define a two dimensional monodromy invariant subspace of the system (\ref{eqn:Phi11})-(\ref{eqn:Phi12}). Note that this function is, up to a multiplication by a function in parameters $\alpha,\beta,\gamma$, equal to an analytic continuation of Horn's $\Phi_1$ function (\cite[Vol.1, \S5.7.1]{ErdelyiEtAl})
\begin{equation}
\Phi_1\left(z,\zeta;\substack{\alpha,\beta\\ \gamma}\right):=\sum_{m,n=0}^\infty\frac{(\alpha)_{m+n}(\beta)_m}{(\gamma)_{m+n}m!n!}z^m\zeta^n.
\end{equation}

\noindent
\change{The other monodromy invariant subspace corresponding to the facet $145$ is spanned by a function (\ref{eqn:2.46}) with function $F\left(z,\zeta;\substack{c_1,c_2\\ c_3}\right)$ given by
\begin{equation}\label{eqn:Phi2}
F\left(z,\zeta;\substack{c_1,c_2\\ c_3}\right)=z^{-c_2}\zeta^{c_2+c_3}\sum_{m_1,m_2=0}^\infty\frac{z^{-m_1}\zeta^{m_1+m_2}}{\Gamma(1-c_1-c_3-m_2)\Gamma(1-c_2-m_1)\Gamma(1+c_2+c_3+m_1+m_2)m_1!m_2!}.
\end{equation}
The series (\ref{eqn:Phi2}) is an entire function in $z^{-1},\zeta$. Note that (\ref{eqn:Phi2}) is a multiple of Horn's $\Phi_2$ series (\cite[Vol.1, \S5.7.1]{ErdelyiEtAl})
\begin{equation}
\Phi_2\left(x,y;\substack{\beta,\beta^\prime\\ \gamma}\right)
=
\sum_{m,n=0}^\infty\frac{(\beta)_{m}(\beta^\prime)_n}{(\gamma)_{m+n}m!n!}x^my^n,
\end{equation}
where $(x,y)=(\zeta/z,\zeta)$ and $\beta,\beta^\prime$ and $\gamma$ are suitable linear combinations of $c_1,c_2$ and $c_3$.
}

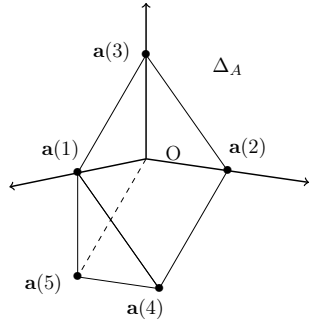
\begin{figure}[H]
\begin{center}
\scalebox{0.7}[0.7]{
\tdplotsetmaincoords{80}{130}
\begin{tikzpicture}[tdplot_main_coords]
\draw[thick,->] (0,0,0) -- (4,0,0);
\draw[thick,->] (0,0,0) -- (0,4,0);
\draw[thick,->] (0,0,0) -- (0,0,3);
\draw (0,0,2) -- (2,0,0) -- (2,2,-2) -- (0,2,0) --cycle;
\draw[-] (2,0,0) -- (2,2,-2) -- (2,0,-2) --cycle;
\draw[-, dashed] (0,0,0) -- (2,0,-2);
\node at (-0.8,0,0){O};
\node at (0,2,2){$\Delta_A$};
\node at (0,2.5,0.5){{\bf a}(2)};
\node at (2.5,0,0.5){{\bf a}(1)};
\node at (1,0,2.2){{\bf a}(3)};
\node at (3,2.5,-2.2){{\bf a}(4)};
\node at (3,0,-2){{\bf a}(5)};
\node at (2,0,0){$\bullet$};
\node at (0,2,0){$\bullet$};
\node at (0,0,2){$\bullet$};
\node at (2,2,-2){$\bullet$};
\node at (2,0,-2){$\bullet$};
\end{tikzpicture}
}
\caption{Newton polytope}
\label{fig:NP}
\end{center}
\end{figure}

\end{exa}

\section{Connection formula associated to a modification}\label{sec:connection}
In this section, we establish a connection formula among particular bases of GG system. Throughout this section, we assume that the matrix $A$ is homogeneous and the column vectors of it generate the lattice $\Z^n.$ Our aim is to construct a path along which we can perform an analytic continuation of the basis discussed in \S2.

\subsection{Connection formula of a Mellin-Barnes integral}\label{subsec:3.2}

We first recall the notion of {\it modification} ({\it perestroika}). Let $\Sigma(A)$ be the secondary polytope, i.e., $\Sigma(A)$ is a convex polytope in $\R^{N}$ of which the dual fan is identical to the secondary fan (\cite[Chapter 7, \S1.D]{GKZbook}). We can conclude that for each regular triangulation $T$, there is a unique vertex $v_T$ of $\Sigma(A)$ such that the normal cone $N_{\Sigma(A)}(v_T)$ of $\Sigma(A)$ at $v_T$ is equal to \change{the closure of} the cone $C_T\subset(\R^{N})^\vee$.
For any pair of regular triangulation $T$ and $T^\prime$, we say $T$ is adjacent to $T^\prime$ if the corresponding vertices $v_T$ and $v_{T^\prime}$ are connected by an edge of $\Sigma(A)$. The adjacency can be interpreted in a combinatorial way. We say $Z\subset\{ 1,\dots,N\}$ is a {\it circuit} if $\{ {\bf a}(i)\}_{i\in Z}$ is a minimal linearly dependent subset of $\{ {\bf a}(j)\}_{j=1}^N$. If $Z$ is a circuit, the corresponding subconfiguration $\{ {\bf a}(i)\}_{i\in Z}$ has only two regular triangulations. They are denoted by $T_+$ and $T_-$. This choice is not canonical and depends on the choice of the generator $u$ of $L_Z=\Ker (A_Z\times:\Z^{Z}\rightarrow \Z^{n})$. If we fix a generator $u$ of $L_Z$, no entry of $u$ is zero by definition. We put $Z_+=\{ i\mid u_i>0\}$ and $Z_-=\{ i\mid u_i<0\}.$ Then $T_+$ (resp. $T_-$) is defined by $\{ Z\setminus\{ i\}\}_{i\in Z_+}$ (resp. $\{ Z\setminus\{ i\}\}_{i\in Z_-}$). A subconfiguration $I\subset\{ 1,\dots,N\}$ is called a {\it corank $1$ configuration} if the rank of $\Ker(A_I\times:\Z^{I}\rightarrow\Z^{n})$ is $1$. We say that a regular polyhedral subdivision $Q$ of $A$ is an {\it almost triangulation} if any refinement of $Q$ is a triangulation. The following propositions are standard (cf.\cite[Chap.7, \S2]{GKZbook}).

\begin{prop}\label{prop:3.1}
A regular polyhedral subdivision $Q$ of $A$ is an almost triangulation if and only if each cell of $Q$ has at most corank $1$ and there is a unique circuit $Z$ such that any corank $1$ cell contains $Z.$
\end{prop}

\begin{prop}\label{prop:3.2}
Let $T$ and $T^\prime$ be a pair of regular triangulations such that $T$ is adjacent to $T^\prime.$ Let $e$ be the edge of $\Sigma(A)$ connecting $v_T$ and $v_{T^\prime}$. Any weight vector $\omega$ in the relative interior of the normal cone $N_{\Sigma(A)}(e)$ of $e$ defines the same regular polyhedral subdivision $S$. Moreover, $S$ is an almost triangulation whose refinements are given by $T$ and $T^\prime.$
\end{prop}

One also has a precise description of the change of adjacent regular triangulations as follows. This is what we call a modification of a regular triangulation. Let $T$ and $T^\prime$ be a pair of adjacent regular triangulations. For any corank $1$ configuration $I$ which contains a circuit $Z$, there are only two triangulations. Namely, they are $T_+=\{ I\setminus\{ i\}\}_{i\in Z_+}$ (resp. $T_-=\{ I\setminus\{ i\}\}_{i\in Z_-}$). Let $Q$ be the intermediate regular polyhedral subdivision of Proposition \ref{prop:3.1}. We decompose $Q$ as $Q=T_{irr}\cup\{ I_s\}_s$ where the irrelevant part $T_{irr}$ consists of simplices and $I_s$ are all corank $1$ configuations. Moreover, all the corank $1$ configurations $I_s$ contain the same circit $Z$ in view of Proposition \ref{prop:3.1} and \ref{prop:3.2}. Since $T$ (or $T^\prime$) is a refinement of $Q$ and $T$ is a triangulation, we see that each $I_s$ has a maximal space dimension $n$, i.e., the convex hull of the origin and the points$\{\{ {\bf a}(i)\}_{i\in I_s}\}$ has a non-zero Euclidian volume. This implies that $|I_s|=n+1$. Thus, if we write $T_+(I_s)$ and $T_-(I_s)$ for the pair of regular triangulations coming from $I_s$, we have $T=T_{irr}\cup \{ T_+(I_s)\}_s$ and $T^\prime=T_{irr}\cup \{ T_-(I_s)\}_s$. This is also denoted by $T=T_{irr}\cup T_+(Z)$ and $T^\prime =T_{irr}\cup T_-(Z)$. Note that $T_{irr}$ can be empty.

\begin{exa}\label{exa:313}
\change{
Let us consider a configuration matrix
$A=
\begin{pmatrix}
1&1&1&1&1\\
0&1&2&0&0\\
0&0&0&1&-1
\end{pmatrix}
$. We fix a basis $\left\{ {}^t(-1,2,-1,0,0),{}^t(-2,0,0,1,1)\right\}$ of $L_A$ through which we identify $L_A^\vee$ with $\Z^2$. The secondary fan and the corresponding regular polyhedral subdivisions are listed in Figure \ref{fig:Mod}. For example, the first quadrant corresponds to the regular triangulation $T_1=\{ 134,135\}$ and the second quadrant corresponds to $T_2=\{ 124,125,234,235\}$. The intermediate regular polyhedral subdivision $Q$ between them is given by $Q=\{ 1234,1235\}$, of which each element is a corank $1$ configuration and the common circuit is $Z=123$. 
}

\begin{figure}[H]
\begin{center}
\scalebox{0.5}[0.5]{
\begin{tikzpicture}
\draw[-] (-6,0) -- (6,0);
\draw[-] (0,0) -- (0,6);
\draw[-] (0,0) -- (-3,-6);
\node at (1.5,4){\scalebox{2}{$T_1=$}};
\node at (2.7,4){$1$};
\node at (4,4.3){$2$};
\node at (5.3,4){$3$};
\node at (3,5.3){$4$};
\node at (3,2.7){$5$};
\node at (3,3){$\bullet$};
\node at (3,5){$\bullet$};
\node at (5,4){$\bullet$};
\node at (3,4){$\bullet$};
\node at (4,4){$\circ$};
\draw (3,4) -- (5,4);
\draw (3,3) -- (5,4) -- (3,5) -- cycle;
\node at (-5,3){$\bullet$};
\node at (-5,5){$\bullet$};
\node at (-3,4){$\bullet$};
\node at (-5,4){$\bullet$};
\node at (-4,4){$\bullet$};  
\node at (-6,4){\scalebox{2}{$T_2=$}};
\draw (-5,3) -- (-5,5) -- (-3,4) -- cycle;
\draw (-4,4) -- (-5,4);
\draw (-4,4) -- (-5,3);
\draw (-4,4) -- (-5,5);
\draw (-3,4) -- (-5,4);
\node at (-6,-5){$\bullet$};
\node at (-6,-3){$\bullet$};
\node at (-4,-4){$\bullet$};
\node at (-6,-4){$\circ$};
\node at (-5,-4){$\bullet$};
\node at (-7,-4){\scalebox{2}{$T_3=$}};
\draw (-6,-5) -- (-6,-3) -- (-4,-4) -- cycle;
\draw (-5,-4) -- (-6,-5);
\draw (-5,-4) -- (-6,-3);
\draw (-5,-4) -- (-4,-4);
\node at (3,-4){$\bullet$};
\node at (3,-2){$\bullet$};
\node at (5,-3){$\bullet$};
\node at (4,-3){$\circ$};
\node at (3,-3){$\circ$};
\node at (1.5,-3){\scalebox{2}{$T_4=$}};
\draw (3,-4) -- (3,-2) -- (5,-3) -- cycle;
\node at (7,0) {$\bullet$};
\node at (7,1){$\bullet$};
\node at (7,-1){$\bullet$};
\node at (9,0){$\bullet$};
\node at (8,0){$\circ$};
\draw (7,-1) -- (9,0) -- (7,1) -- cycle;
\node at (0,8){$\bullet$};
\node at (-1,8){$\bullet$};
\node at (-1,7){$\bullet$};
\node at (-1,9){$\bullet$};
\node at (1,8){$\bullet$};
\draw (-1,9) -- (-1,7) -- (1,8) -- cycle;
\draw (-1,8) -- (1,8);
\node at (-9,0) {$\bullet$};
\node at (-9,1){$\bullet$};
\node at (-9,-1){$\bullet$};
\node at (-7,0){$\bullet$};
\node at (-8,0){$\bullet$};
\draw (-9,-1) -- (-9,1) -- (-7,0) -- cycle;
\draw (-9,1) -- (-8,0);
\draw (-9,-1) -- (-8,0);
\draw (-8,0) -- (-7,0);
\node at (-3,-7){$\bullet$};
\node at (-4,-7){$\circ$};
\node at (-4,-6){$\bullet$};
\node at (-4,-8){$\bullet$};
\node at (-2,-7){$\bullet$};
\draw (-4,-6) -- (-4,-8) -- (-2,-7) -- cycle;
\end{tikzpicture}
}
\caption{The secondary fan and the corresponding regular polyhedral subdivisions}
\label{fig:Mod}
\end{center}
\end{figure}
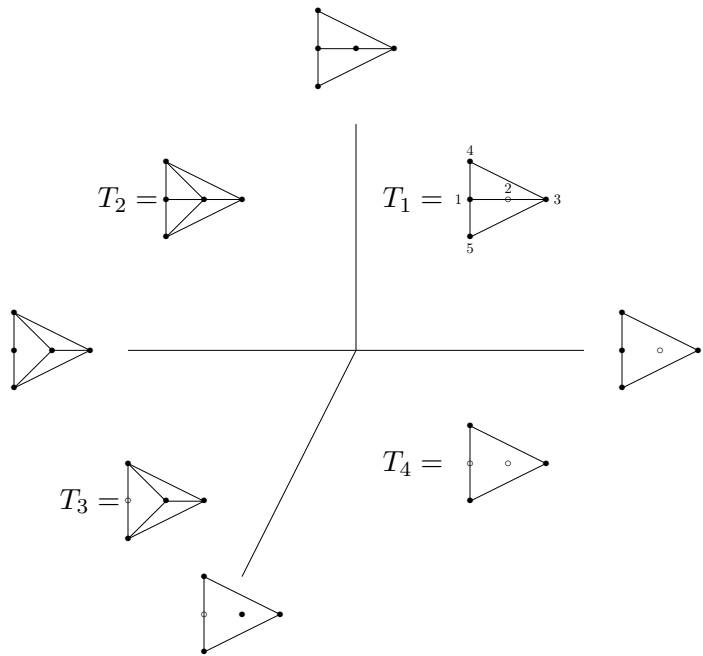

\end{exa}

\vspace{2mm}

Now, let us concentrate on the GG system coming from a corank $1$ configuration. Namely, we concentrate on the GG sytem $GG(A_I)$ for a corank $1$ configuration $I$. \change{By exchanging indices and omitting the suffix $I$ of $A_I$,} we write $A=\left({\bf a}(1)|\dots|{\bf a}(n+1)\right)$ for a corank $1$ configuration which may not generate the ambient lattice $\Z^{n}$ but span the vector space $\Q^n$. This is a tentative notation and used only in the rest of this subsection.
We put $I=\{ 1,\dots,n+1\}$. Let us take a generator $u$ of the lattice $L_A$. We put $I_{\geq 0}=\{ j\in I\mid u_j\geq 0\},$ $I_{0}=\{ j\in I\mid u_j= 0\},$ $Z_{+}=\{ j\in I\mid u_j>0 \}$, and $Z_-=\{ j\in I\mid u_j<0\}$. Note that the circuit $Z$ contained in $I$ is given by $Z=Z_+\cup Z_-$. We fix an element $j_0\in Z_+$ and put $\s=I\setminus\{ j_0\}.$  Let $\{ {\bf e}_i\}_{i\in\s}$ be the standard basis of the lattice $\Z^\s$. We set ${\bf 1}_{-}=\sum_{i\in Z_-}{\bf e}_i$. Consider an integral 

\begin{equation}\label{eqn:3.1}
I_\s(z_I;c)=\frac{1}{2\pi\ii}\int_{C}\frac{\Gamma(-s)\displaystyle\prod_{i\in Z_-}\Gamma\left(p_{\s i}(c+{\bf a}(j_0)s)\right)}{\displaystyle\prod_{i\in\s\cap I_{\geq 0}}\Gamma\left(1-p_{\s i}(c+{\bf a}(j_0)s)\right)}(e^{\pi\ii{\bf 1}_{-}}z_\s)^{-A_\s^{-1}(c+{\bf a}(j_0)s)}(e^{\pi\ii}z_{j_0})^sds,
\end{equation}
where $C$ is a vertical contour from $-\ii\infty$ to $+\ii\infty$ separating two spirals of poles of Gamma functions in the integrand (\cite[Chapter 4]{SlaterBook}). Note that $(e^{\pi\ii{\bf 1}_{-}}z_\s)^{A_\s^{-1}c}I_\s(z_I;c)$ depends only on circuit variables $z_+:=(z_j)_{j\in Z_+}$ and $z_-:=(z_j)_{j\in Z_-}$. By Stirling's formula, one can easily prove that this integral is convergent if 
\begin{equation}\label{eqn:3.2}
|\arg\left( (e^{\pi\ii{\bf 1}_{-}}z_\s)^{-A_\s^{-1}{\bf a}(j_0)}(e^{\pi\ii}z_{j_0})\right)|<\pi.
\end{equation}

We rewrite the convergence condition (\ref{eqn:3.2}). Observe that $
\begin{pmatrix}
-A_\s^{-1}{\bf a}(j_0)\\
1
\end{pmatrix}
=
u_{j_0}^{-1}u
.$ Therefore, we have
\begin{equation}
|\arg\left( (e^{\pi\ii{\bf 1}_{-}}z_\s)^{-A_\s^{-1}{\bf a}(j_0)}(e^{\pi\ii}z_{j_0})\right)|=u_{j_0}^{-1}|\arg (e^{\pi\ii{\bf 1}_{-}}z_\s,e^{\pi\ii}z_{j_0})\cdot
u|,
\end{equation}
where the symbol $\cdot$ denotes the dot product. Thus, the convergence condition can be written as 
\begin{align}
 &-\pi<\sum_{i\in Z_-}\arg(e^{\pi\ii}z_i)u_{j_0}^{-1}u_i+\sum_{j\in Z_+\cap\s}\arg(z_j)u_{j_0}^{-1}u_j+\arg(e^{\pi\ii}z_{j_0})<\pi\\
\iff&-2\pi u_{j_0}<\sum_{i\in Z_-}\arg(e^{\pi\ii}z_i)u_i+\sum_{j\in Z_+}\arg(z_j)u_j<0.
\end{align}
Note that the last condition depends only on the circuit variables $(z_j)_{j\in Z}$.

We consider an analytic continuation of (\ref{eqn:3.1}) via Mellin-Barnes contour throw. This method has been previously discussed by several authors in various settings (\cite[Chap.4]{SlaterBook}). We first \change{note} that $(e^{\pi{\bf 1}_{-}}z_\s)^{A_\s^{-1}c}I_\s(z_I;c)$ is a univariate function of a complex variable
\begin{equation}
\zeta=(e^{\pi{\bf 1}_{-}}z_\s)^{-A_\s^{-1}{\bf a}(j_0)}e^{\pi\ii}z_{j_0}=e^{\pi\ii}(e^{\pi{\bf 1}_{-}}z_{-})^{u_{j_0}^{-1}u_-}(z_{+})^{u_{j_0}^{-1}u_+},
\end{equation}
where $u_+$ (resp. $u_-$) is the vector consisting of entries of $u$ labeled by the set $Z_+$ (resp. $Z_-$). The distribution of poles are as in the Figure \ref{fig:2}. If we evaluate the integral along the poles $s=0,1,2,\dots$, we have
$I_\s(z_I;c)=\psi_{\s\cap I_{\geq 0}}^{Z_-}(z;c)$.

\begin{figure}[h]
\begin{center}
\begin{tikzpicture}
\draw (0,0) node[below right]{O}; 
\draw[thick, ->] (-4,0)--(3.2,0) node[right]{$\re s$} ; 
\draw[thick, ->] (0,-1.2)--(0,3.2) node[above]{$\im s$} ; 
\node at (0,0){$\times$};
\node at (1,0){$\times$};
\node at (2,0){$\times$};
\node at (1,-0.5){$1$};
\node at (2,-0.5){$2$};
\node at (3,-0.5){$\cdots$};
\node at (-1,2){$\times$};
\node at (-2.5,2){$\times$};
\node at (-4,2){$\times$};
\node at (-5,2){$\cdots$};
\node at (-0.6,1.5){\scalebox{0.7}{$\frac{u_{j_0}}{u_i}p_{\s i}(c)$}};
\node at (-2.3,1.5){\scalebox{0.7}{$\frac{u_{j_0}}{u_i}(p_{\s i}(c)+1)$}};
\node at (-1,-1){$\times$};
\node at (-2,-1){$\times$};
\node at (-3,-1){$\times$};
\node at (-4,-1){$\cdots$};
\end{tikzpicture}
\caption{distribution of poles}
\label{fig:2}
\end{center}
\end{figure}
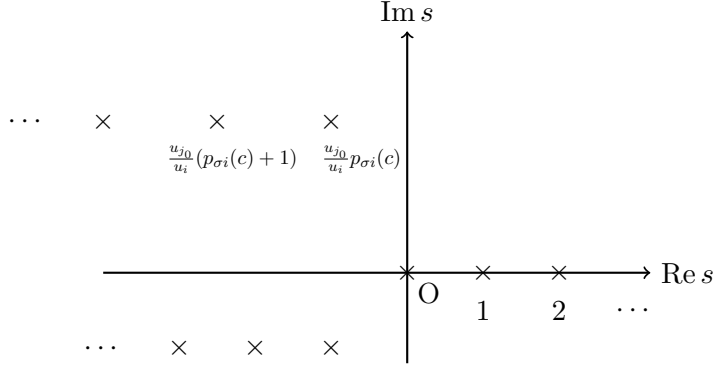

Let us fix an element $i_0\in Z_-$ and consider the evaluation of $I_\s(z_I;c)$ along a negative spiral $s=\frac{u_{j_0}}{u_{i_0}}(p_{\s i_0}(c)+m)$ ($m\in\Z_{\geq 0}$). We put $\s^\prime=I\setminus\{ i_0\}$. We take a vector $v_\s\in\C^{n+1}$ so that the $j_0$-th entry is $0$ and the equality $Av=-c$ is true. Note that $v_{\s}$ is uniquely determined. In view of the equality $
A_{\s^\prime}^{-1}Av_\s=-A_{\s^\prime}^{-1}c$, we obtain relations
\begin{equation}\label{EqC}
p_{\s^\prime j_0}({\bf a}(i_0))\cdot p_{\s i_0}(c)=p_{\s^\prime i}(c)
\end{equation}
and
\begin{equation}\label{EqD}
p_{\s^\prime i}({\bf a}(i_0))\cdot p_{\s i_0}(c)+p_{\s i}(c)=p_{\s^\prime i}(c)\quad\quad (i\in\s^\prime\setminus\{ j_0\}).
\end{equation}


\noindent
A direct computation employing (\ref{EqC}) and (\ref{EqD}) shows that
\begin{equation}
I_\s(z_I;c)=\sum_{i\in Z_-}\frac{1}{p_{\s i}({\bf a}(j_0))}\psi_{I_{\geq 0}\setminus\{ j_0\}}^{(Z_-\setminus\{ i\})\cup\{ j_0\}}(z_I;c).\label{EVA}
\end{equation}

\noindent
In the same way, we choose any $\tilde{\bf k}_\s=(\tilde{k}_i)_{i\in\s}\in \Z^{\s}$ and consider a function $I_\s(e^{2\pi\ii\tilde{\bf k}_\s}z_\s,z_{j_0};c)$. The integral is convergent if

\begin{equation}\label{CommonSector}
-2\pi u_{j_0}<-2\pi u_{j_0}\sum_{i\in\s}\tilde{k}_i p_{\s i}({\bf a}(j_0))+\sum_{i\in Z_-}\arg(e^{\pi\ii}z_i)u_i+\sum_{j\in Z_+}\arg(z_j)u_j<0.
\end{equation}

\noindent
By a direct computation, we have a formula
\begin{equation}\label{eqn:3.11}
I_\s(e^{2\pi\ii\tilde{\bf k}_\s}z_\s,z_{j_0};c)=\displaystyle\sum_{i\in Z_-}\frac{1}{p_{\s i}({\bf a}(j_0))}\psi_{I_{\geq 0}\setminus\{ j_0\},(\tilde{\bf k}_{\s\setminus\{ i\}},\overset{j_0}{\breve{0}})}^{(Z_-\setminus\{ i\})\cup\{ j_0\}}(z_I;c).
\end{equation}


\subsection{Some lemmata needed for the construction of the path of analytic continuation}\label{subsec:3.1}
Suppose that $T$ and $T^\prime$ are adjacent regular triangulations of the configuration matrix $A$. We inherit the notation of the previous subsection. First, we prove a
\begin{prop}\label{prop:CS}
Let $Q$ be the intermediate almost triangulation and $Z$ be the common circuit. Then, one can choose a complete system of representatives $\{ \tilde{\bf k}_\s \}$ of $\Z^{\s}/\Z {}^tA_\s$ for any corank $1$ configuration $I$ in $Q$ and for any $\s=I\setminus\{ j_0\}$ with $j_0\in Z_+$ so that the inequalities (\ref{CommonSector}) define a non-empty open subset in the space of circuit variables $\C^Z$.
\end{prop}

\begin{proof}

We set ${\bf 1}_{\s}=\sum_{i\in\s}{\bf e}_i\in\Z^\s$. Then, the equivalence class $[{\bf 1}_{\s}]$ in the group $\Z^\s/\Z {}^tA_\s$ is zero. Indeed, since $A$ is homogeneous, for any element $j\in\bs$, we have ${}^t{\bf 1}_{\s}A_\s^{-1}{\bf a}(j)=|A_\s^{-1}{\bf a}(j)|=1$. This computation combined with the fact that a pairing $\Z^{\s}/\Z{}^tA_\s\times\Z^{n}/\Z A_\s\ni (u,v)\mapsto {}^tuA_\s^{-1}v\in\Q/\Z$ is perfect and that $\Z^{n}/\Z A_\s=\sum_{j\in\bs}\Z[{\bf a}(j)]$ implies that $[{\bf 1}_{\s}]=0$. Therefore, for any $\tilde{\bf k}_\s\in\Z^{n}/\Z{}^tA_\s$ and $r\in\Z$, we have $[\tilde{\bf k}_\s]=[\tilde{\bf k}_\s+r{\bf 1}_{\s}]$. On the other hand, if we replace $\tilde{\bf k}_\s$ by $\tilde{\bf k}_\s+r{\bf 1}_{\s}$, the sector defined by the inequality (\ref{CommonSector}) is translated by $2\pi u_{j_0}\times r$. Thus, we can choose a complete system of representatives as in the statement of the proposition. \change{For example, we can choose $\tilde{\bf k}_\s=(\tilde{k}_i)_{i\in\s}$ so that $0\leq \sum_{i\in\s}\tilde{k}_i p_{\s i}({\bf a}(j_0))<1$.}
\end{proof}

Proposition \ref{prop:CS} gives the information of the argument of the path of analytic continuations. Now, we control other directions $z_{\bar{I}}$ for a corank $1$ configuration in $Q$. The important point is that, if we consider a modification, there can be several corank $1$ configurations in general. So we have to control $z_{\bar{I}}$ simultaneously. Let us identify the space of row vectors $\R^{1\times N}$ with the dual lattice $(\R^N)^\vee$ via dot product. We put 
\begin{equation}
\tilde{C}_{I,+}=\{ \omega\in\R^{1\times N}\mid \omega_{I\setminus \{ j\}}A_{I\setminus \{ j\}}^{-1}{\bf a}(k)<\omega_k,\text{ for any }j\in Z_+\text{ and }k\in\bar{I}\}
\end{equation}
and
\begin{equation}
\tilde{C}_+=\bigcap_{I:\text{corank }1\text{ configurations in }Q}\tilde{C}_{I,+}.
\end{equation}

\begin{prop}\label{prop:Ctilde}
$C_Q\cup C_T\subset\tilde{C}_+.$
\end{prop}

\begin{proof}
Let us take an element $\omega\in C_T$. Since we have $I\setminus\{ j\}\in T$ for any $j\in Z_+$, there exists a row vector ${\bf n}\in\Q^{1\times n}$ such that
\begin{equation}\label{eqn:314}
\begin{cases}
{\bf n}\cdot {\bf a}(i)=\omega_i& (i\in I\setminus\{ j\})\\
{\bf n}\cdot {\bf a}(k)<\omega_k& (k\in \bar{I}\cup\{ j\}).
\end{cases}
\end{equation}
We regard $\omega_{I\setminus\{ j\}}:=(\omega_i)_{i\in I\setminus\{ j\}}$ as a  row vector. From the first equality of (\ref{eqn:314}), we can derive the equality ${\bf n}=\omega_{I\setminus\{ j\}}A_{I\setminus\{ j\}}^{-1}.$ Substituting this equality to the inequality above, for any $k\in \bar{I}\cup\{ j\}$, we obtain 
\begin{equation}
\omega_{I\setminus \{ j\}}A_{I\setminus \{ j\}}^{-1}{\bf a}(k)<\omega_k.
\end{equation}
This implies the inclusion $C_T\subset \tilde{C}_+$.

Next, we take $\omega\in C_Q$. Since $I\in Q$, there exists a row vector ${\bf n}\in \Q^{1\times n}$ such that
\begin{equation}
\begin{cases}
{\bf n}\cdot {\bf a}(i)=\omega_i& (i\in I)\\
{\bf n}\cdot {\bf a}(k)<\omega_k& (k\in \bar{I}).
\end{cases}
\end{equation}
Again, from the first equality, we can derive the equality ${\bf n}=\omega_{I\setminus\{ j\}}A_{I\setminus\{ j\}}^{-1}$ for any $j\in Z_+$. Therefore, from the second inequality, for any $j\in Z_+$ and $k\in \bar{I}$, we obtain
\begin{equation}
\omega_{I\setminus \{ j\}}A_{I\setminus \{ j\}}^{-1}{\bf a}(k)<\omega_k.
\end{equation}
This implies the inclusion $C_Q\subset \tilde{C}_+$.
\end{proof}

\change{
\begin{rem}
It is easy to see that the cone $\tilde{C}_{I,+}$ is a pull-back of a cone in $L_A^\vee\otimes_{\Z}\R$, i.e., we have an identity $\tilde{C}_{I,+}=\pi_A^{-1}(\pi_A(\tilde{C}_{I,+}))$. Indeed, the kernel of $\pi_A$ is spanned by row vectors of the form $v\cdot A$ for some row vector $v$ of length $n$ and $\tilde{C}_{I,+}$ is invariant under the translation by these vectors.
\end{rem}
}
\change{
\begin{exa}
Let $(w_1,w_2)$ be the linear coordinate of $L_A^\vee\otimes_{\Z}\R$ specified by the choice of a basis of $L_A$ as in Example \ref{exa:313}. We consider three examples of modifications: from $T_1$ to $T_2$, from $T_3$ to $T_2$ and from $T_4$ to $T_1$. We obtain the following table.
\begin{center}
\begin{tabular}{c|c|c|c|c}
modifications&corank 1 configurations&$Z$&$Z_+$&$\pi_A(\tilde{C}_+)$\\
\hline
$T_1\rightarrow T_2$&1234,1235&123&2&$\{ w_2>0\}$\\
\hline
$T_3\rightarrow T_2$&1245&145&1&$\{w_2>2w_1\}$\\
\hline
$T_4\rightarrow T_1$&1345&145&1&$\{w_2<2w_1\}$
\end{tabular}
\end{center}
Note that we have $\pi_A(\tilde{C}_+)=C_{T_1}\cup C_{T_2}$, $\pi_A(\tilde{C}_+)\supsetneq C_{T_2}\cup C_{T_3}$ and $\pi_A(\tilde{C}_+)\subsetneq C_{T_1}\cup C_{T_4}$.
\end{exa}
}

We conclude this section with a lemma on the existence of a good direction.

\begin{prop}
For any corank $1$ configuration $I$ in $Q$, there is a weight vector $\omega^{(I)}=(0_I,\omega_{\bar{I}})$ which belongs to the cone $C_Q$ such that $\omega_{\bar{I}}>0.$ 
\end{prop}

\begin{proof}
We take a weight vector $\omega$ in the relative interior of the cone $C_Q.$ Since $I$ belongs to $Q$, there is a row vector ${\bf n}\in\Q^{1\times n}$ such that
\begin{equation}
\begin{cases}
{\bf n}\cdot {\bf a}(i)=\omega_i& (i\in I)\\
{\bf n}\cdot {\bf a}(k)<\omega_k& (k\in \bar{I}).
\end{cases}
\end{equation}
If we put $\omega^{(I)}=\omega-{\bf n}A$, this satisfies the desired properties.
\end{proof}

\begin{cor}\label{cor:3.5}
There exists a weight vector $\omega_Q$ in the cone $C_Q$ such that $\omega_Q=(0_{\text{core}(Q)},\omega_{\overline{\text{core}(Q)}})$ and $\omega_{\overline{\text{core}(Q)}}>0$. Here, $\text{core}(Q)=\displaystyle\bigcap_{I:\text{corank }1\text{ configurations in }Q}I\subset\{ 1,\dots,N\}$ and $\overline{\text{core}(Q)}=\{ 1,\dots,N\}\setminus \text{core}(Q)$.
\end{cor}

\noindent
\change{Note that we have an inclusion $Z\subset{\rm core}(Q)$.}

\subsection{Estimate of a difference operator of infinite order}\label{subsec:3.3}

In this subsection, we consider a hypergeometric function of the following type:
\begin{equation}
F\left(\substack{{\bf a},{\bf b}\\ \tilde{\bf a},\tilde{\bf b}};z\right)=\sum_{n=0}^\infty
\frac{
\Gamma({\bf a}+{\bf b}n)
}
{
\Gamma(\tilde{\bf a}+\tilde{\bf b}n)
}z^n.\label{eqn:3.18}
\end{equation}
Here, parameters are ${\bf a}\in\C^p$, $\tilde{\bf a}\in\C^q$, ${\bf b}\in\Z^p_{\geq 0}$, $\tilde{\bf b}\in\Z^q_{\geq 0}$ and we assume that for any $n\in\Z_{\geq 0}$, ${\bf a}+{\bf b}n$ has no non-positive integer entry. Note that if ${\bf  b}\in\Q_{\geq 0}^p$ or $\tilde{\bf b}\in\Q^q_{\geq 0}$, there is a positive integer $k$ so that $k{\bf b}\in\Z^p_{\geq 0},k\tilde{\bf b}\in\Z^q_{\geq 0}$. Therefore, we have a decomposition
\begin{equation}
F\left(\substack{{\bf a},{\bf b}\\ \tilde{\bf a},\tilde{\bf b}};z\right)=\sum_{i=0}^{k-1}z^iF\left(\substack{{\bf a}+i{\bf b},k{\bf b}\\ \tilde{\bf a}+i\tilde{\bf b},k\tilde{\bf b}};z^k\right).
\end{equation}
Thus, the consideration is reduced to the case when ${\bf b}\in\Z_{\geq 0}^p$ and $\tilde{\bf b}\in\Z^q_{\geq 0}$. Now we consider the following operator which is sometimes referred to as Erd\'elyi-Kober integral:
\begin{equation}
(I_0^{(\alpha,\beta),\kappa}f)(z)=\frac{1}{\Gamma(\beta)}\int_0^1t^{\alpha-1}(1-t)^{\beta-1}f(t^{\kappa}z)dt\quad\quad\left( \re\alpha,\re\beta>0,\kappa\in\R_{\geq 0}\right).
\end{equation}
Here, $f(z)$ is a germ of a univariate holomorphic function defined around the origin. In terms of power series, this operator is well-understood. Indeed, if we write $f(z)=\sum_{n=0}^\infty a_nz^n$, we have
\begin{align}
(I_0^{(\alpha,\beta),\kappa}f)(z)&=\frac{1}{\Gamma(\beta)}\sum_{n=0}^\infty a_n\left(\int_0^1t^{\alpha+\kappa n-1}(1-t)^{\beta-1}dt\right)z^n\nonumber\\
&=\sum_{n=0}^\infty a_n\frac{\Gamma(\alpha+\kappa n)}{\Gamma(\alpha+\beta+\kappa n)}z^n\label{eqn:Shift}
\end{align}
From the formula (\ref{eqn:Shift}), shifting the complex parameters ${\bf a}, \tilde{\bf a}$ in (\ref{eqn:3.18}) can be expressed in terms of the integral operator $I_0^{(\alpha,\beta),\kappa}$. We want to generalize this operator $I_0^{(\alpha,\beta),\kappa}$ to $\alpha\in\C\setminus\Z$, $\beta\in\C$, when $\kappa\in\Z_{\geq 0}$.

\noindent
\underline{case 1: $\alpha\in\C\setminus\Z$ and $\beta\in\C\setminus\Z$}

In this case, we have
\begin{equation}
(I_0^{(\alpha,\beta),\kappa}f)(z)=\frac{1}{\Gamma(\beta)(1-e^{-2\pi\ii\alpha})(1-e^{-2\pi\ii\beta})}\int_Pt^{\alpha-1}(1-t)^{\beta-1}f(t^{\kappa}z)dt,
\end{equation}
where $P$ is the Pochhammer cycle connecting $t=0$ and $t=1$ (\cite[CHAP.X, 12$\cdot$43]{WW}).

\noindent
\underline{case 2: $\alpha\in\C\setminus\Z$ and $\re\beta>0$}

In this case, we have
\begin{equation}
(I_0^{(\alpha,\beta),\kappa}f)(z)=\frac{1}{\Gamma(\beta)(1-e^{-2\pi\ii\alpha})}\int_Qt^{\alpha-1}(1-t)^{\beta-1}f(t^{\kappa}z)dt,
\end{equation}
where $Q$ is the cycle which starts from and ends at 1 as in \change{Figure \ref{fig:4}}.

\begin{figure}[h]
\center
\begin{tikzpicture}
\draw (0,0) node[below right]{O};
\node at (0,0) {$\cdot$};
\node at (2,0) {$\cdot$};
\node at (2.3,-0.3) {$1$};
\draw[->-=.5,domain=90:270] plot ({cos(\x)}, {sin(\x)});
\draw[ ->-=.6] (2,0)--(0,1);
\draw[ ->-=.4] (0,-1)--(2,0);
\node at (1,-0.5) {$\bullet$};
\node at (1.5,-0.8) {$\arg t=0$};
\end{tikzpicture}
\caption{cycle $Q$}
\label{fig:4}
\end{figure}

\noindent
\underline{case 3: $\alpha\in\C\setminus\Z$ and $\beta=-s$ ($s\in\Z_{\geq 0}$)}

In this case, we first take $\beta$ to be a generic complex number and $\re\alpha>0$
\begin{equation}
(I_0^{(\alpha,\beta),\kappa}f)(z)=\frac{1}{\Gamma(\beta)(1-e^{-2\pi\ii\beta})}\int_Rt^{\alpha-1}(1-t)^{\beta-1}f(t^{\kappa}z)dt,
\end{equation}
where $R$ is the cycle which starts from and ends at $0$ as in \change{Figure \ref{fig:5}}.

\begin{figure}[h]
\center
\begin{tikzpicture}
\draw (0,0) node[below right]{O};
\node at (0,0) {$\cdot$};
\node at (2,0) {$\cdot$};
\node at (2.3,-0.3) {$1$};
\draw[->-=.5,domain=90:-90] plot ({2+cos(\x)}, {sin(\x)});
\draw[ ->-=.6] (0,0)--(2,1);
\draw[ ->-=.6] (2,-1)--(0,0);
\node at (1,0.5) {$\bullet$};
\node at (0.8,0.8) {$\arg t=0$};
\end{tikzpicture}
\caption{cycle $R$}
\label{fig:5}
\end{figure}

\noindent
If we let $\beta$ tend to $-s$, we obtain
\begin{equation}
(I_0^{(\alpha,-s),\kappa}f)(z)=\frac{(-1)^{s+1}s!}{2\pi\ii}\oint_{\partial\Delta(1;\varepsilon)}\frac{t^{\alpha-1}f(t^{\kappa}z)}{(1-t)^{s+1}}dt.
\end{equation}
Here, $\Delta(1;\varepsilon)$ denotes a disk around $1$ with radius $0<\ve<1$. The formula above is valid even when $\alpha\in\C\setminus\Z$.

Now we are in a position to apply the integral representation of the difference operator to the key Gevrey estimate. For any $0<\varepsilon<R, 0<\varepsilon_\theta$, we put $S_{\varepsilon,\varepsilon_\theta,R}=\overline{\Delta(0;\varepsilon)}\cup\{ |\arg z-\theta|\leq\varepsilon_\theta,|z|\leq R\}.$

\begin{lem}\label{Lemma1}
Let $a\in\C$ and a finite subset $\{ h_s\}_s\subset\Q$ satisfy $a+\sum_{s}l_sh_s\notin\Z$ for any non-negative integers $l_s$ and let $\kappa\in\Z_{\geq 0}$. Then, for any $0<\varepsilon^\prime<\varepsilon$, $0<\varepsilon_\theta^\prime<\varepsilon_\theta$, $0<R^\prime<R$, there exist constants $0<C$ and $0<r_s$ which only depend on $a,h_s,\kappa,\varepsilon^\prime,\varepsilon^\prime_\theta,R^\prime$ such that for any holomorphic function $f$ in a neighbourhood of $S_{\varepsilon, \varepsilon_\theta,R}$, the inequality
\begin{equation}
|(I_0^{(a+\sum_sl_sh_s,-\sum_sl_sh_s),\kappa}f)(z)|\leq C\prod_sr_s^{l_s}|h_sl_s|^{l_sh_s}\sup\{ |f(z)|\mid z\in S_{\varepsilon,\varepsilon_\theta,R}\}
\end{equation}
holds for any $z\in S_{\varepsilon^\prime,\varepsilon_\theta^\prime,R^\prime}$ and $l_s\in\Z_{\geq 0}$.
\end{lem}

\begin{proof}
By homotopy, we can choose integration contours $P,Q,\partial\Delta(1;\tilde{\varepsilon})$ so that for any $t$ in one of these contours and for any $z\in S_{\varepsilon^\prime,\varepsilon_\theta^\prime,R^\prime}$, we have $t^\kappa z\in S_{\varepsilon,\varepsilon_\theta,R}$. By the assumption, we have $a+\sum_sl_sh_s\notin\Z$. Therefore, we have
\begin{align}
&(I_0^{(a+\sum_sl_sh_s,-\sum_sl_sh_s),\kappa}f)(z)\\
=&
\begin{cases}
\frac{e^{-\sum_s\pi\ii l_sh_s}\Gamma(1+\sum_sl_sh_s)}{2\pi\ii(1-e^{-2\pi\ii(a+\sum_sl_sh_s)})}\int_Pt^{a+\sum_sl_sh_s-1}(1-t)^{-\sum_sl_sh_s-1}f(t^\kappa z)dt &(\sum_sl_sh_s\notin\Z)\\
\frac{1}{\Gamma(-\sum_sl_sh_s)(1-e^{-2\pi\ii(a+\sum_sl_sh_s)})}\int_Qt^{a+\sum_sl_sh_s-1}(1-t)^{-\sum_sl_sh_s-1}f(t^\kappa z)dt&(-\sum_sl_sh_s\in\Z_{>0})\\
\frac{(-1)^{\sum_sl_sh_s-1}(\sum_sl_sh_s)!}{2\pi\ii}\oint_{\partial\Delta(1;\tilde{\varepsilon})}\frac{t^{a+\sum_sl_sh_s-1}f(t^\kappa z)}{(1-t)^{\sum_sl_sh_s+1}}dt&(-\sum_sl_sh_s\in\Z_{\leq 0}).
\end{cases}
\end{align}
Since each $h_s$ is rational, we see that the desired estimate exists for the first and the third case since $t^{\sum_sl_sh_s}(1-t)^{-\sum_sl_sh_s}=\prod_s(t^{h_s}(1-t)^{-h_s})^{l_s}$. As for the second case, if we put $r=\sup\{|1-t|\mid t\in Q\}$ and $r^\prime=\inf\{ |t|\mid t\in Q\}$,  we have, for any small positive number $\delta,$ an estimate 
\begin{align}
&\frac{1}{\Gamma(-\sum_sl_sh_s)(1-e^{-2\pi\ii(a+\sum_sl_sh_s)})}\int_Qt^{a+\sum_sl_sh_s-1}(1-t)^{-\sum_sl_sh_s-1}f(t^\kappa z)dt\\
\leq&\frac{C}{\Gamma(-\sum_sl_sh_s)}{r^\prime}^{\sum_sl_sh_s}r^{-\sum_sl_sh_s-\delta}\sup\{ |f(z)|\mid z\in S_{\varepsilon,\varepsilon_\theta,R}\}
\end{align}
\end{proof}

In the same way, one can prove a
\begin{lem}\label{Lemma2}
Let $\tilde{a}\in\C\setminus\Z$, $\{\tilde{h}_s\}\subset\Q$ be arbitrary and $\tilde{\kappa}\in\Z_{\geq 0}$. Then, for any $0<\varepsilon^\prime<\varepsilon$, $0<\varepsilon_\theta^\prime<\varepsilon_\theta$, $0<R^\prime<R$, there exist constants $0<\tilde{C}$ and $0<\tilde{r}$ which only depend on $\tilde{a},\tilde{h}_s,\tilde{\kappa},\varepsilon^\prime,\varepsilon^\prime_\theta,R^\prime$ such that for any holomorphic function $f$ in a neighbourhood of $S_{\varepsilon, \varepsilon_\theta,R}$, the inequality
\begin{equation}
|(I_0^{(\tilde{a},\sum_sl_s\tilde{h}_s),\tilde{\kappa}}f)(z)|\leq \tilde{C}\prod_s\tilde{r}^{l_s}|\tilde{h}_sl_s|^{-\tilde{h}_sl_s}\sup\{ |f(z)|\mid z\in S_{\varepsilon,\varepsilon_\theta,R}\}
\end{equation}
holds for any $z\in S_{\varepsilon^\prime,\varepsilon_\theta^\prime,R^\prime}$ and any $l_s\in\Z_{\geq 0}$.
\end{lem}

\begin{proof}
In this case, we have
\begin{equation}
(I_0^{(\tilde{a},\sum_sl_s\tilde{h}_s),\tilde{\kappa}}f)(z)=
\begin{cases}
\frac{e^{-\pi\ii \sum_sl_s\tilde{h}_s}\Gamma(1-\sum_sl_s\tilde{h}_s)}{2\pi\ii(1-e^{-2\pi\ii\tilde{a}})}\int_Pt^{\tilde{a}-1}(1-t)^{\sum_sl_s\tilde{h}_s-1}f(t^{\tilde{\kappa}} z)dt &(\sum_sl_s\tilde{h}_s\notin\Z)\\
\frac{1}{\Gamma(\sum_sl_s\tilde{h}_s)(1-e^{-2\pi\ii\tilde{a}})}\int_Qt^{\tilde{a}-1}(1-t)^{\sum_sl_s\tilde{h}_s-1}f(t^{\tilde{\kappa}} z)dt&(\sum_sl_s\tilde{h}_s\in\Z_{>0})\\
\frac{(-1)^{\sum_sl_s\tilde{h}_s-1}(\sum_sl_s\tilde{h}_s)!}{2\pi\ii}\oint_{\partial\Delta(1;\tilde{\varepsilon})}\frac{t^{\tilde{a}-1}f(t^{\tilde{\kappa}} z)}{(1-t)^{\sum_sl_s\tilde{h}_s+1}}dt&(\sum_sl_s\tilde{h}_s\in\Z_{\leq 0}).
\end{cases}
\end{equation}
The rest of the proof is similar to that of Lemma \ref{Lemma1}.
\end{proof}

By applying the lemmata above repeatedly, we obtain the desired

\begin{thm}\label{GevreyEstimate}
Let $\{{\bf h}_s\}_s\subset\Q^p$, $\{\tilde{\bf h}_s\}_s\subset\Q^q$ be finite subsets such that for any $l_s,n\in\Z_{\geq 0}$, ${\bf a}+\sum_sl_s{\bf h}_s+{\bf b}n$ and $\tilde{\bf a}$ does not have any integer entry. If $F\left(\substack{{\bf a},{\bf b}\\ \tilde{\bf a},\tilde{\bf b}};z\right)$ is holomorphic in a neighbourhood of $S_{\varepsilon,\varepsilon_\theta,R}$. Then, for any $0<\varepsilon^\prime<\varepsilon$, $0<\varepsilon_\theta^\prime<\varepsilon_\theta$, $0<R^\prime<R$, there exist constants $0<C$ and $0<r_s$ which only depend on ${\bf a},\tilde{\bf a},{\bf h}_s,\tilde{\bf h}_s,\varepsilon^\prime,\varepsilon^\prime_\theta,R^\prime$ such that
\begin{equation}
|F\left(\substack{{\bf a}+\sum_sl_s{\bf h}_s,{\bf b}\\ \tilde{\bf a}+\sum_sl_s\tilde{\bf h}_s,\tilde{\bf b}};z\right)|\leq C\prod_sr_s^{l_s}(||{\bf h}_s|-|\tilde{\bf h}_s||\cdot l_s)^{(|{\bf h}_s|-|\tilde{\bf h}_s|)l_s}\sup\{ |F\left(\substack{{\bf a},{\bf b}\\ \tilde{\bf a},\tilde{\bf b}};z\right)|\mid z\in S_{\varepsilon,\varepsilon_\theta,R}\}
\end{equation}
holds for any $z\in S_{\varepsilon^\prime,\varepsilon_\theta^\prime,R^\prime}$ and any $l_s\in\Z_{\geq 0}$. Here, we set $0^0:=1$.
\end{thm}

\begin{proof}
\change{If we write ${\bf h}_s=(h_{s1},\dots,h_{sp})$ and $\tilde{\bf h}_s=(\tilde{h}_{s1},\dots,\tilde{h}_{sq})$,} we have an equality
\begin{align}
F\left(\substack{{\bf a}+\sum_sl_s{\bf h}_s,{\bf b}\\ \tilde{\bf a}+\sum_sl_s\tilde{\bf h}_s,\tilde{\bf b}};z\right)&=\prod_{i=1}^pI_0^{(a_i+\sum_sl_sh_{si},-\sum_sl_sh_{si}),b_i}\circ\prod_{j=1}^q
I_0^{(\tilde{a}_j,\sum_sl_s\tilde{h}_{sj}),\tilde{b}_j}\left[F\left(\substack{{\bf a},{\bf b}\\ \tilde{\bf a},\tilde{\bf b}};z\right)\right].
\end{align}
Theorem follows from a successive application of Lemma \ref{Lemma1} and Lemma \ref{Lemma2}.
\end{proof}

\subsection{Construction of a path and a proof of a connection formula}\label{subsec:conn}

\begin{figure}[h]
\center
\begin{tikzpicture}
\draw (0,0) node[below right]{O};
\draw[-=.5] (0,0)--(6,1);
\draw[-=.5] (0,0)--(8,8);
\draw[-=.5] (0,0)--(1,6);
\draw[-=.5] (3,2)--(9,3);
\draw[-=.5] (3,2)--(8,7);
\draw[-=.5] (2,3)--(7,8);
\draw[-=.5] (2,3)--(3,9);
\draw[-=.5,red] (3,3)--(9,3);
\draw[-=.5,red] (3,3)--(6,9);
\node at (2,1){$C_T$};
\draw[->=.5,domain=10:45] plot({(1.9)*cos(\x)},{(1.9)*sin(\x)});
\draw[<-=.5,domain=10:45] plot({(1.9)*cos(\x)},{(1.9)*sin(\x)});
\draw[->=.5,domain=45:80] plot({(1.9)*cos(\x)},{(1.9)*sin(\x)});
\draw[<-=.5,domain=45:80] plot({(1.9)*cos(\x)},{(1.9)*sin(\x)});
\draw[->=.5,domain=-6:94] plot({(0.7)*cos(\x)},{(0.7)*sin(\x)});
\draw[<-=.5,domain=-6:94] plot({(0.7)*cos(\x)},{(0.7)*sin(\x)});
\node at (1,2){$C_{T^\prime}$};
\node at (8.5,8.5){$C_Q$};
\node at (5,5){\textcolor{red}{$r\omega_Q+\tilde{C}_+$}};
\node at (7,4){$\omega_T+C_T$};
\node at (4,7){$\omega_{T^\prime}+C_{T^\prime}$};
\draw[->-=.5] (5,2.5)--(3,4.5);
\draw[->-=.5] (8,5.5)--(6,7.5);
\draw[->-=.8,dotted] (5,2.5)--(8,5.5);
\draw[->-=.8,dotted] (3,4.5)--(6,7.5);
\node at (4.2,3.7){$\gamma$};
\node at (8,6.5){$r\omega_Q+\gamma$};
\draw[-=.5] (0,0)--(8,-1);
\draw[-=.5] (0,0)--(-1,8);
\node at (1,0.5){$C_{T_{irr}}$};
\end{tikzpicture}
\caption{cones}
\end{figure}
Let $T$ and $T^\prime$ be adjacent regular triangulations. We use the notation of \S3.1. For any $z=(z_1,\dots)\in(\C^*)^N$, we set $-\log|z|:=(-\log|z_1|,\dots)\in\R^N$. \change{We first take a vector $\omega_T\in C_T$ (resp. $\omega_{T^\prime}\in C_{T^\prime}$) so that $\omega_T+C_T\subset-\log|U_T|$ (resp. $\omega_{T^\prime}+C_{T^\prime}\subset-\log|U_{T^\prime}|$) is true and fix a point $z_{start}\in\C^N$ so that $-\log|z_{start}|\in(\omega_{T}+C_{T})$. Then, we take a point $ z_{end}\in(\C^*)^N$ so that $-\log|z_{end}|\in (\omega_{T^\prime}+C_{T^\prime}\cap \tilde{C}_+)$.} Note that Proposition \ref{prop:Ctilde} and the fact that $\tilde{C}_+$ is an open set implies that $C_{T^\prime}\cap \tilde{C}_+\neq\varnothing$. We take a path $\gamma(t)$ $(0\leq t\leq 1)$ in $(\C^*)^N$ \change{so that the argument of each entry of it is fixed and $-\log|\gamma(t)|$ is a straight line.} \change{In order to clarify the choice of the argument,} we choose a complete system of representatives $\{\tilde{\bf k}_\s\}$ of $\Z^\s/\Z{}^tA_\s$ for any $\s=I\setminus\{ j_0\}$ and $j_0\in Z_+$ as in Proposition \ref{prop:CS}. Then, we choose $\arg z$ along this path $\gamma$ so that the inequalities (\ref{CommonSector}) are valid. When $z$ runs over this path $\gamma$, the circuit variables \change{are contained in} a set of the form $S_{\varepsilon, \varepsilon_\theta,R}$. \change{We retake the path $\gamma$ so that $-\log|\gamma|$ is translated by a vector $\omega_Q=({\bf 0}_{core(Q)},\omega_{\overline{core(Q)}})\in C_Q$ as in Corollary \ref{cor:3.5} and the arguments remain unchanged.} Therefore, by Theorem \ref{GevreyEstimate} combined with Proposition \ref{prop:Ctilde}, \change{and by replacing $\omega_Q$ by $r\omega_Q$ for a large positive number $r$ if necessary,} for any corank $1$ configuration $I$ in $Q$, and for any $j_0\in Z_+$ \change{and $\tilde{\bf k}_{\s}$}, the function

\begin{equation}\label{eqn:series}
\prod_{j\in\bar{I}}\mathfrak{D}_jI_\s(e^{2\pi\ii\tilde{\bf k}_\s}z_\s,z_{j_0};c)
\end{equation}
with $\s=I\setminus\{ j_0\}$ \change{is} convergent on $\gamma$. \change{Note that the function $I_\s(z_\s,z_{j_0};c)$ takes the form $I_\s(z_\s,z_{j_0};c)=(e^{\pi\ii{\bf 1}_-}z_{\s})^{-A_\s^{-1}c}\tilde{I}_\s(z_-,z_+;c)$ where $(z_-,z_+)$ is the circuit variable and the function (\ref{eqn:series}) with $\tilde{\bf k}_\s={\bf 0}$ takes the form
\begin{equation}
(e^{\pi\ii{\bf 1}_-}z_{\s})^{-A_\s^{-1}c}\sum_{{\bf m}=(m_k)_k\in\Z_{\geq 0}^{\bar{I}}}\frac{\tilde{I}_\s(z_-,z_+;c+A_{\bar{I}}{\bf m})}{{\bf m!}}\prod_{k\in\bar{I}}\left( (e^{\pi\ii{\bf 1}_-}z_{\s})^{-A_\s^{-1}{\bf a}(k)}z_k\right)^{m_k}.
\end{equation}
The absolute value of $z_{\s}^{-A_\s^{-1}{\bf a}(k)}z_k$ is small when $-\log|z|$ is in a far translation of $\tilde{C}_+$ inside itself.} Therefore, the analytic continuation of (\ref{eqn:series}) from $z_{start}$ to $z_{end}$ is computable from the formula (\ref{eqn:3.11}).

In \change{summary}, we obtain the following
\begin{thm}\label{thm:main}
Let $T$ and $T^\prime$ be adjacent regular triangulations of $A$. Suppose that for any corank $1$ configuration $I$ appearing in the modification of $T$ and $T^\prime$ and for any $j_0\in Z_+$, a complete set of representatives $\{\tilde{\bf k}_\s\}$ of $\Z^\s/\Z {}^tA_\s$ with $\s=I\setminus\{ j_0\}$ is given as in Proposition \ref{prop:CS}. Then, along the path $\gamma$ constructed above, for any corank $1$ configuration $I$ and $j_0\in Z_+$, we have a connection formula 
\begin{equation}
\psi_{I_{\geq 0}\setminus\{ j_0\},\tilde{\bf k}_\s}^{Z_-}(z;c)
=
\displaystyle\sum_{i\in Z_-}\frac{1}{p_{\s i}({\bf a}(j_0))}\psi_{I_{\geq 0}\setminus\{ j_0\},(\tilde{\bf k}_{\s\setminus\{ i\}},\overset{j_0}{\breve{0}})}^{(Z_-\setminus\{ i\})\cup\{ j_0\}}(z;c).
\end{equation}
Moreover, $\Gamma$-series corresponding to $\s\in T_{irr}$ are invariant after analytic continuation.
\end{thm}

\change{
\begin{rem}
Fixing $c$ very generic with respect to $T$ and $T^\prime$, Theorem \ref{thm:main} provides a connection formula of the corresponding GKZ system $M_A(c)$.
\end{rem}
}

\change{
\begin{exa}
Let us illustrate the general construction above in terms of the configuration matrix $A$ of Example \ref{exa:313} and the modification $T_3\rightarrow T_2$. The combinatorial data necessary for the connection formula are $I=1245$, $Z=145$, $Z_+=1$, $Z_-=45$. The intermediate polyhedral subdivision is given by $Q=\{ 1245, 234,235\}$. We take a primitive generator ${}^t(2,0,-1,-1)$ of $L_{A_{1245}}$. Let us consider a complete system of representatives $\{ \tilde{\bf k}_0:={}^t(0,0,0),\tilde{\bf k}_1:={}^t(0,1,0)\}$ of $\Z^{\{2,4,5\}}/\Z {}^tA_{245}$. Then, the inequalities (\ref{CommonSector}) are equivalent to a single inequality $0<\arg z_1-\frac{1}{2}(\arg z_4+\arg z_5)<\pi$. We fix the arguments of $z_1,\dots,z_5$ so that this inequality is true. As ${\rm core}(Q)=I=1245$, we can take a vector $\omega_Q=(0,0,1,0,0)$ as that of Corollary \ref{cor:3.5}. Note that $\pi_A(\omega_Q)=(-1,0)$ under the identification $L_A^\vee\otimes\R\simeq\R^2$ of Example \ref{exa:313}. Let us take $z_{start}$ (resp. $z_{end}$) so that $-\log|z_{start}|$ (resp. $-\log|z_{end}|$) is in a far translation of $C_{T_3}$ (resp. $C_{T_2}$) inside itself. We take a path $\gamma$ so that $-\log|\gamma|$ is a straight line from $-\log|z_{start}|$ to $-\log|z_{end}|$. If we consider the boundary value to the coordinate plane $\{ z_3=0\}$, we obtain a connection formula 
\begin{equation}
{\rm bv}_3\left(\psi_{2,\tilde{\bf k}_0}^{45}(z;c)\right)\rightsquigarrow 2\left\{ {\rm bv}_3\left(\psi_{2,{\bf 0}}^{14}(z;c)\right)+{\rm bv}_3\left(\psi_{2,{\bf 0}}^{15}(z;c)\right)\right\}
\end{equation}
and
\begin{equation}
{\rm bv}_3\left(\psi_{2,\tilde{\bf k}_1}^{45}(z;c)\right)\rightsquigarrow 2\left\{ e^{-2\pi\ii c_3}{\rm bv}_3\left(\psi_{2,{\bf 0}}^{14}(z;c)\right)+{\rm bv}_3\left(\psi_{2,{\bf 0}}^{15}(z;c)\right)\right\}
\end{equation}
along a path $\gamma^\prime$ obtained from $\gamma$ by truncating the $3$rd entry. This formula can be obtained from the intermediate Mellin-Barnes integral
\begin{equation}
I(z_1,z_4,z_5;c):=\frac{1}{2\pi\ii}\int_C\frac{\Gamma(-s)\Gamma\left(\frac{1}{2}(c_1-c_2+c_3+s)\right)\Gamma\left(\frac{1}{2}(c_1-c_2-c_3+s)\right)}{\Gamma(1-c_2)}\left( (z_4z_5)^{-\frac{1}{2}}z_1\right)^sds.
\end{equation}
where $C$ is the integration contour separating the positive and the negative spirals of the poles of the integrand. For example, ${\rm bv}_3\left(\psi_{2,\tilde{\bf k}_0}^{45}(z;c)\right)$ is given by $e^{\pi\ii(c_2-c_1)}z_2^{-c_2}z_4^{\frac{1}{2}(c_2-c_1-c_3)}z_5^{\frac{1}{2}(c_2+c_3-c_1)}I(z_1,z_4,z_5;c)$ and ${\rm bv}_3\left(\psi_{2,\tilde{\bf k}_1}^{45}(z;c)\right)$ is given by $e^{\pi\ii(2c_2-2c_1-c_3)}z_2^{-c_2}z_4^{\frac{1}{2}(c_2-c_1-c_3)}z_5^{\frac{1}{2}(c_2+c_3-c_1)}I(z_1,e^{2\pi\ii}z_4,z_5;c)$. Note that $I(z_1,z_4,z_5;c)$ depends only on the circuit variables $z_1,z_4,z_5$. We want to estimate the series
\begin{equation}\label{eqn:342}
\mathfrak{D}_3\circ{\rm bv}_3\left(\psi_{2,\tilde{\bf k}_0}^{45}(z;c)\right)=
e^{\pi\ii(c_2-c_1)}z_2^{-c_2}z_4^{\frac{1}{2}(c_2-c_1-c_3)}z_5^{\frac{1}{2}(c_2+c_3-c_1)}\sum_{m=0}^\infty\frac{I(z_1,z_4,z_5;c+m{\bf a}(3))}{m!}(-z_1z_2^{-2}z_3)^m.
\end{equation}
If we replace $-\log|\gamma|$ by $-\log|\gamma|+r_1\omega_Q$ for some large number $r_1>0$, we see that the absolute values of $z_1,z_2,z_4,z_5$ remain unchanged while the absolute value of $z_3$ (hence that of $z_1z_2^{-2}z_3$) become smaller as $r_1$ increases. Theorem \ref{GevreyEstimate} shows that there is a positive number $r_2$ so that the ratio $\frac{I(z_1,z_4,z_5;c+m{\bf a}(3))}{m!}$ is bounded by a constant multiple of the power $r_2^m$ along $\gamma$. Thus, if we replace the contour $\gamma$ so that $-\log|\gamma|$ is translated by $r_1\omega_Q$ for some $r_1>0$, (\ref{eqn:342}) is convergent along $\gamma$. As we have the same estimate for the solution $\mathfrak{D}_3\circ{\rm bv}_3\left(\psi_{2,\tilde{\bf k}_1}^{45}(z;c)\right)$, we obtain connection formulas
\begin{equation}
\psi_{2,\tilde{\bf k}_0}^{45}(z;c)\rightsquigarrow 2\left\{ \psi_{2,{\bf 0}}^{14}(z;c)+\psi_{2,{\bf 0}}^{15}(z;c)\right\}
\end{equation}
and
\begin{equation}
\psi_{2,\tilde{\bf k}_1}^{45}(z;c)\rightsquigarrow 2\left\{ e^{-2\pi\ii c_3}\psi_{2,{\bf 0}}^{14}(z;c)+\psi_{2,{\bf 0}}^{15}(z;c)\right\}
\end{equation}
along $\gamma$. On the other hand, $-\log|\gamma|$ is contained in the cone $C_{T_{irr}}=C_{T_2}\cup C_{T_3}$ and we can conclude that the other solutions $\psi_{234,{\bf 0}}(z;c)$ and $\psi_{235,{\bf 0}}(z;c)$ remain unchanged along $\gamma$.
\end{exa}
}

\begin{exa}
\change{Let us discuss yet another example. We} consider a $4\times 6$ matrix 
$
A=
\begin{pmatrix}
1&0&0&0&1&1\\
0&1&0&0&1&0\\
0&0&1&0&0&1\\
0&0&0&1&-1&-1
\end{pmatrix}.
$ The secondary fan is a complete fan in $(\R^6)^\vee$. The projected image of $\Sigma(A) $ through the projetion $\pi_A:(\R^6)^\vee\rightarrow L_A^\vee\otimes_\Z\R$ is shown in Figure \ref{fig:SecFan}. Here, we use the isomorphism $L_A\simeq\Z^2$ specified by choosing a basis $\left\{{}^t(0,-1,1,0,1,-1),{}^t(-1,0,-1,1,0,1)\right\}$ of $L_A$.

\begin{figure}[H]
\begin{center}
\scalebox{0.8}[0.8]{
\begin{tikzpicture}
\draw (0,0) node[anchor= south west]{O}; 
\draw[-] (0,0)--(2,0);
\draw[-] (0,0)--(0,2);
\draw[-] (0,0)--(-2,2);
\draw[-] (0,0)--(-2,0);
\draw[-] (0,0)--(0,-2);
\draw[-] (0,0)--(2,-2);
\node at (1,1){$T$};
\node at (1.7,-0.8){$T^\prime$};
\end{tikzpicture}
}
\end{center}
\caption{The projected image of the secondary fan}
\label{fig:SecFan}
\end{figure}
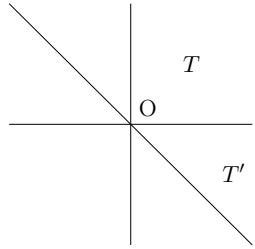

\noindent
The basis of solutions at $T=\{1234,1236,1256\}$ is given as follows:
\begin{align}
\phi_{T,1234}(z;c)=&e^{-\pi\ii(c_1+c_3)}\frac{\sin\pi c_2 \Gamma(c_1)\Gamma(c_2)\Gamma(c_3)}{\pi\Gamma(1-c_4)}z_1^{-c_1}z_2^{-c_2}z_3^{-c_3}z_4^{-c_4}F_1\left( \substack{c_1,c_2,c_3\\ 1-c_4};\frac{z_4z_5}{z_1z_2},\frac{z_4z_6}{z_1z_3}\right)\\
\phi_{T,1236}(z;c)=&e^{-\pi\ii(c_1+c_3+2c_4)}\frac{\sin\pi c_2\sin\pi (-c_4)\Gamma(c_1+c_4)\Gamma(c_3+c_4)\Gamma(c_2)\Gamma(-c_4)}{\pi^2}\times\nonumber\\
&z_1^{-c_1-c_4}z_2^{-c_2}z_3^{-c_3-c_4}z_6^{c_4}G_2\left( c_1+c_4,c_2,-c_4,c_3+c_4;-\frac{z_4z_6}{z_1z_3},-\frac{z_3z_5}{z_2z_6}\right)\\
\phi_{T,1256}(z;c)=&\frac{\sin\pi(c_1+c_4)\sin\pi(c_2+c_3+c_4)\sin\pi c_3\Gamma(c_1+c_3)\Gamma(c_2+c_3+c_4)\Gamma(c_3)}{\pi^3\Gamma(1+c_3+c_4)}\times\nonumber\\
&z_1^{-c_1-c_4}z_2^{-c_2-c_3-c_4}z_5^{c_3+c_4}z_6^{-c_3}F_1\left( \substack{c_3,c_1+c_4,c_2+c_3+c_4\\ 1+c_3+c_4};\frac{z_3z_5}{z_2z_6},\frac{z_4z_5}{z_1z_2}\right)
\end{align}

\noindent
where $F_1\left(\substack{\alpha,\beta,\beta^\prime\\ \gamma};z_1,z_2\right)$ and $G_2\left(\alpha,\alpha^\prime, \beta, \beta^\prime;z_1,z_2\right)$ are Appell's $F_1$ and Horn's $G_2$ series defined by
\begin{equation}
F_1\left(\substack{\alpha,\beta,\beta^\prime\\ \gamma};z_1,z_2\right)=\sum_{m,n=0}^\infty \frac{(\alpha)_{m+n}(\beta)_m(\beta^\prime)_n}{(\gamma)_{m+n}m!n!}z_1^mz_2^n
\end{equation}
and
\begin{equation}
G_2\left(\alpha,\alpha^\prime, \beta, \beta^\prime;z_1,z_2\right)=\displaystyle\sum_{m,n=0}^\infty\frac{(\alpha)_m(\alpha^\prime)_n(\beta)_{n-m}(\beta^\prime)_{m-n}}{m!n!}z_1^mz_2^n.
\end{equation}
Note that these functions $\phi_{T,\s}(z;c)$ are convergent when $z_4$ and $z_5$ are small enough and other variables $z_1,z_2,z_3,z_6$ are fixed.
On the other hand, the basis of solutions at $T^\prime=\{ 1246,2346,1256\}$ is given as follows:
\begin{align}
\phi_{T^\prime,1246}(z;c)=&e^{-\pi\ii c_3}\frac{\sin\pi(c_1+c_4)\sin\pi c_2\Gamma(c_1)\Gamma(c_2)\Gamma(c_1+c_4)}{\pi\sin\pi (c_3-c_1)\Gamma(1+c_1-c_3)}\times\nonumber\\
&z_2^{-c_2}z_3^{c_1-c_3}z_4^{-c_1-c_4}z_6^{-c_1}F_1\left( \substack{c_1+c_4,c_2,c_1\\ 1+c_1-c_3};\frac{z_1z_3}{z_4z_6},\frac{z_3z_5}{z_2z_6}\right)\\
\phi_{T^\prime,2346}(z;c)=&e^{-\pi\ii c_1}\frac{\sin\pi c_2\sin\pi (c_3+c_4)\Gamma(c_1-c_3)\Gamma(c_2)\Gamma(c_3)\Gamma(c_3+c_4)}{\pi^2}\times\nonumber\\
&z_1^{c_3-c_1}z_2^{-c_2}z_4^{-c_3-c_4}z_6^{-c_3}G_2\left( c_3,c_2,c_1-c_3,c_3+c_4;-\frac{z_1z_3}{z_4z_6},-\frac{z_4z_5}{z_1z_2}\right)\\
\phi_{T^\prime,1256}(z;c)=&\frac{\sin\pi(c_1+c_4)\sin\pi(c_2+c_3+c_4)\sin\pi c_3\Gamma(c_1+c_3)\Gamma(c_2+c_3+c_4)\Gamma(c_3)}{\pi^3\Gamma(1+c_3+c_4)}\times\nonumber\\
&z_1^{-c_1-c_4}z_2^{-c_2-c_3-c_4}z_5^{c_3+c_4}z_6^{-c_3}F_1\left( \substack{c_3,c_1+c_4,c_2+c_3+c_4\\ 1+c_3+c_4};\frac{z_3z_5}{z_2z_6},\frac{z_4z_5}{z_1z_2}\right)
\end{align}

We see that the modification from $T$ to $T^\prime$ is controlled by the corank $1$ configuration $12346$ and that $Z_+=46$, $Z_-=13$ and $I_0=2$. In view of Theorem \ref{thm:main}, solving the boundary value problem for $\{ z_5=0\}$ yields the connection formulae
\begin{equation*}
\phi_{T,1234}(z;c)\leadsto \phi_{T^\prime,1246}(z;c)+\phi_{T^\prime,2346}(z;c)
\end{equation*}
and
\begin{equation*}
\phi_{T,1236}(z;c)\leadsto e^{-\pi\ii c_4}\frac{\sin\pi c_1}{\sin\pi(c_1+c_4)}\phi_{T^\prime,1246}(z;c)+e^{-\pi\ii c_4}\frac{\sin\pi c_3}{\sin\pi(c_3+c_4)}\phi_{T^\prime,2346}(z;c).
\end{equation*}
Note that the connection coefficients are translation invariant, i.e., they belong to the field $\C(T_1,T_2,T_3,T_4)$. On the other hand, we have
\begin{equation*}
\phi_{T,1256}(z;c)=\phi_{T^\prime,1256}(z;c).
\end{equation*}
\end{exa}

\section*{Acknowledgement}
The author thanks Nobuki Takayama for valuable discussions. This work is supported by JSPS KAKENHI Grant Number 19K14554 and JST CREST Grant Number JP19209317 including AIP challenge program.

\bibliographystyle{alpha}
\bibliography{myreferences}

\end{document}